\documentclass[11pt,twoside]{amsart} 

  \usepackage{url}
  \usepackage{amssymb}

  \usepackage[a4paper, margin=2.5cm]{geometry}

  \input xypic
  \xyoption{all}

 \usepackage{times}
 
  \usepackage{amsmath, amsthm, amsfonts}

\newtheorem{theorem}{Theorem}[section]
\newtheorem{lemma}[theorem]{Lemma}

\theoremstyle{definition}
\newtheorem{definition}[theorem]{Definition}
\newtheorem{example}[theorem]{Example}

\newtheorem{proposition}[theorem]{Proposition}
\newtheorem{corollary}[theorem]{Corollary}

\theoremstyle{remark}
\newtheorem{remark}[theorem]{Remark}

\numberwithin{equation}{section}

  \renewcommand{\cH}{{\mathcal H}}
  
  \newcommand{\cK}{{\mathcal K}}

  \newcommand{\cE}{{\mathcal E}}
  
   \newcommand{\cN}{{\mathcal N}}
  \renewcommand{\cL}{{\mathcal L}}
  
  \newcommand{\cP}{{\mathcal P}}
  \newcommand{\cU}{{\mathcal U}}
   \newcommand{\cV}{{\mathcal V}}
    \newcommand{\cW}{{\mathcal W}}
  \newcommand{\cG}{{\mathcal G }}
  \newcommand{\cC}{{\mathcal C }}

\newcommand{\cT}{{\mathcal T }}
 \newcommand{\ccR}{{\mathcal R}}

\newcommand{\XX}{{\mathfrak X}}

   \newcommand{\ba}{\begin{eqnarray}}
   \newcommand{\na}{\end{eqnarray}}
   \newcommand{\ban}{\begin{eqnarray*}}
   \newcommand{\nan}{\end{eqnarray*}}

  \newcommand{\C}{{\mathbb C}}
  \newcommand{\R}{{\mathbb R}}
  \newcommand{\Z}{{\mathbb Z}}
  \newcommand{\Q}{{\mathbb Q}}

  \newcommand{\PP}{{\mathbb P}}

  \renewcommand{\a}{\alpha}
  \renewcommand{\b}{\beta}
   
  \renewcommand{\c}{\gamma}

  \newcommand{\<}{\langle}
  \renewcommand{\>}{\rangle}

  \newcommand{\T}{{\mathfrak T}}

    \newcommand{\disp}{\displaystyle}

  \def\cancel#1#2{\ooalign{$\hfil#1\mkern1mu/\hfil$\crcr$#1#2$}}
\def\Dirac{\mathpalette\cancel D}

\begin{document}

\title[Delocalized Chern character for stringy orbifold K-theory]{Delocalized Chern character for stringy orbifold K-theory}



  \author[Jianxun Hu] {Jianxun Hu}
  \address{Department of Mathematics\\  Sun Yat-Sen University\\
                        Guangzhou,  510275\\ China }
  \email{stsjxhu@mail.sysu.edu.cn}
 
  \author[Bai-Ling Wang]{Bai-Ling Wang}
  \address{Department of Mathematics\\
  Australian National University\\
  Canberra ACT 0200 \\
  Australia}
  \email{bai-ling.wang@anu.edu.au}


\date{}

\dedicatory{}

\begin{abstract}
 In this paper, we define a stringy product on $K^*_{orb}(\XX) \otimes \C $, the orbifold  K-theory of any  almost complex  presentable  orbifold $\XX$. We establish   that  under this stringy product, the  delocalized Chern character  
     \[
   ch_{deloc}  :  K^*_{orb}(\XX) \otimes \C  \longrightarrow   H^*_{CR}(\XX), 
   \]
   after a canonical modification,  is a ring isomorphism.  Here $ H^*_{CR}(\XX)$ is the Chen-Ruan cohomology of $\XX$.
The  proof relies on   an   intrinsic description of  the  obstruction bundles  in the construction of the Chen-Ruan product.     As an application, we investigate this stringy product on the equivariant K-theory  
    $K^*_G(G)$ of a finite group   $G$ with the conjugation action. It turns out that the stringy product  is different from  the Pontryagin product (the latter is also called the fusion product in string theory). 

\end{abstract}

\maketitle


\section{Introduction}

The notion of  orbifold   was first introduced
by Satake under the  name  V-manifold.   There have been many very interesting 
developments since its inception.  For example, Kawasaki's orbifold index theory has been applied
extensively in the study of geometric quantizations, and in the development of orbifold string 
theory in  quantum physics.

 Recall that an orbifold $\XX$  is a paracompact Hausdorff space $X$ equipped with a  compatible system of   orbifold atlases   locally modeled on quotient spaces of   Euclidean spaces 
 by finite group actions.  
  For each $x\in X$, there is a neighbourhood $U_x$ and  a homeomorphism $U_x \cong \tilde U_x/G_x$. $\XX$ is called effective if  each  local group $G_x$ acts on   $\tilde U_x$ effectively.  For any  orbifold $\XX$, the  de Rham cohomology  $H^*_{dR}(\XX)$ is well defined, and by a theorem of Satake \cite{Sat}, is isomorphic to the singular cohomology of the underlying space 
  $X= |\XX|$.

 Associated to an effective  orbifold $\XX$,  we have  a canonical non-effective orbifold, called the inertia orbifold  $  \tilde \XX$ of $\XX$.    The inertia orbifold $\tilde \XX$ consists of connected components of different dimensions, see page 7 in \cite{CR}, 
\[
\tilde \XX = \bigsqcup_{(g)} \XX_{(g)}
\]
where $(g) \in \cT_1$, the set of equivalence classes of conjugacy classes in  local groups. 
Each $\XX_{(g)}$ is called a twisted sector of $\XX$,  and is a sub-orbifold of $\XX$.
The underlying topological space of $\tilde \XX$, denoted by $|\tilde \XX|$,  is the disjoint union of $X$ and  the singularity set  
\[
\Sigma X = \{ (x, (g) ) | x\in X, G_x \neq \{1\},  (g)  \in \text{Conj}(G_x) \},
\]
where $\text{Conj}(G_x)$ denotes the set of conjugacy classes in $G_x$.

In the development of Gromov-Witten theory for  symplectic orbifolds, Chen and Ruan  in \cite{CR} discovered a remarkable
new cohomology theory on any almost    complex orbifold $\XX$,  called the Chen-Ruan
cohomology $H^*_{CR} (\XX)$.  Almost complex orbifolds are those  orbifolds with local models  given by a finite group acting unitarily on complex spaces. 
The  Chen-Ruan
cohomology $H^*_{CR} (\XX)$, as a classical limit of  an   orbifold quantum cohomology,   is  a cohomology of  the inertia orbifold $\tilde\XX$
\[
H^*_{CR} (\XX) =  (H^*_{dR}(\tilde \XX, \C), \circ_{CR}) 
\]
with a  new product  $\circ_{CR}$,  utilising the obstruction bundles over the moduli spaces of stable pseudo-holomorphic  orbifold curves in $\XX$.    The obstruction bundle $E^{[2]}$  in the construction of the Chen-Ruan product  is a complex  orbifold vector bundle  over 
$\tilde \XX^{[2]} =\tilde \XX \times_{e} \tilde \XX$, where $e: \tilde \XX \to \XX$ is the immersion 
defined by the sub-orbifold structure on each connected component of $\tilde\XX$.
The associativity  of the Chen-Ruan product follows from a  property of the obstruction bundles discovered by
Chen-Ruan in \cite{CR} using the gluing construction in Gromov-Witten theory.



For a compact  almost complex orbifold $\XX$, Adem, Ruan and Zhang in \cite{ARZ} defined a stringy product   on $K_{orb}^*(\tilde \XX, \tau)$, the twisted   K-theory of the inertia orbifold   $\tilde\XX$ with  a transgressive  twisting $\tau$.  This product will be called the 
Adem-Ruan-Zhang product, denoted by  $\bullet_{ARZ}$.  For an orbifold  $\XX$ arising from  a smooth, projective variety  $M$ with an action of a finite group $G$  or   a Deligne-Mumford stack, an analogous product $\bullet_{JKK}$   was defined by Jarvis, Kaufmann and Kimura
in \cite{JKK} on  the untwisted  orbifold K-theory of     $\tilde\XX$.
 The ring  $(K_{orb}^*(\tilde \XX), \bullet_{JKK})$  is called the  full orbifold K-theory in \cite{JKK}.   This  full orbifold K-theory   was generalized to  any  compact  abelian quotient  orbifold  $\XX$  by Becerra and Uribe  in \cite{BecUri}. The associated stringy product will be called the Becerra-Uribe produt, denoted by $ \bullet_{BU}$. They also   established an isomorphism
\[
(K_{orb}^*(\tilde \XX), \bullet_{BU}) \cong  (K_{orb}^*(\tilde \XX), \bullet_{ARZ}).
\]
   In \cite{BecUri}, 
Becerra and Uribe also  established   a ring homomorphism between the  
orbifold K-theory of $\tilde\XX$   and the Chen-Ruan cohomology  
under  a modified  Chern character map 
\[
\widetilde{ch}:  (K^*_{orb}(\tilde \XX), \bullet_{BU})   \longrightarrow ( H_{CR}^*(\XX), \circ_{CR}).
\]
We remark  that this Chern character map  is not an isomorphism over the complex coefficients.
The full orbifold K-theory $(K_{orb}^*(\tilde \XX), \bullet_{JKK})$  was further studied by Goldin, Harada, Holm and Kimura for abelian symplectic quotients in \cite{GHHK}. They gave a complete description of the ring structure of the full orbifold K-theory of weighted
projective spaces obtained as symplectic quotients of $\C^n$ by weighted $S^1$-actions.

It is known that there is a   delocalized   Chern character (Cf.  \cite{BC2} for proper actions
   of discrete groups and \cite{BryNis} for \'etale groupoids)
   \[
    ch_{deloc}:  K^*_{orb}(\XX)  \longrightarrow H^*(\tilde \XX, \C) =  \bigoplus_{(g)}  H^*( \XX_{(g)}, \C)
    \]
    relating the orbifold K-theory  $K^*_{orb}(\XX)$  to the 2-periodic de Rham cohomology of the inertia   orbifold $\tilde\XX$. 
    In this paper, we define a stringy product on the orbifold K-theory  $K^*_{orb}(\XX)$ of an almost complex compact orbifold $\XX$.     
The main result of this paper is that after a canonical modification,  $ch_{deloc}$ is an isomorphism over the complex coefficients, and   sends the stringy product on $K^*_{orb}(\XX)$ to the Chen-Ruan  cup product on $H^*_{dR}(\tilde \XX, \C)$.

 The construction of this modified  $ch_{deloc}$ relies on an intrinsic description of 
the obstruction bundle for the Chen-Ruan product and  Adem-Ruan-Zhang's stringy product.  Associated to the orbifold  immersion 
\[
e=\bigsqcup_{(g)} e_{(g)}:  \tilde \XX=  \bigsqcup_{(g)} \XX_{(g)} \to \XX, 
\]
there is a 2-sector orbifold $\tilde\XX^{[2]} =  \tilde\XX\times_e \tilde \XX $, which consists of a disjoint union of 
 sub-orbifolds of  $\XX$
 \[
\bigsqcup_{(g_1, g_2)} \XX_{(g_1, g_2)}
\]
each of which is labelled by an equivalence  class of conjugacy pairs in local groups. Consider  the following commutative diagram 
   \[    \xymatrix{   \XX_{(g_1, g_2)} \ar[rr]^{e_1}\ar[dd]_{e_2} \ar[dr]^{e_{12} } && 
    \XX_{(g_1)}  \ar[dd]^{e_{(g_1)}} \\   &   \XX_{(g_1g_2)}   \ar[dr]^{e_{(g_1g_2)}} & \\    \XX_{(g_2)}  \ar[rr]^{e_{(g_2)}} &&  \XX.   } 
\]
Then the obstruction bundle $E^{[2]}$ is a complex orbifold vector bundle over $\tilde\XX^{[2]}$, whose component over each $\XX_{(g_1, g_2)} $ is denoted by $E^{[2]}_{(g_1, g_2)}$ (see \cite{CR} for its definition).  

  

Over the inertia orbifold $\tilde \XX$, there is an orbifold  complex vector bundle 
\[
\cN_e = \bigsqcup_{(g)}  \cN_{(g)}  \longrightarrow  \bigsqcup_{(g)}  \XX_{(g)}  
\]
given by the orbifold normal bundle   of  each $\XX_{(g)}$ in $\XX$, each of which admits an automorphism of finite order.  We can choose a  Hermitian metric  on
the tangent bundle of  $\XX$ such that  the automorphism $\Phi$ acts unitarily on $\cN_e$. Then each   $\cN_{(g)} $ has an eigen-bundle  decomposition 
  \[
  \cN_{(g)}   =    \bigoplus_{\theta_{(g)} \in \Q\cap  (0, 1) } \cN (\theta_{(g)}) 
  \]
  where $\Phi$ on $ \cN  (\theta_{(g)}) $ is multiplication by $e^{2\pi \sqrt{-1}\theta_{(g)}}$ with $\theta_{(g)} \in \Q\cap  (0, 1)$.   Define
  \ba\label{cN_g:decomp}
    \cN_{e, \Phi} =  \bigoplus_{(g)}  \sum_{\theta_{(g)}}  \theta_{(g)}   \cN  (\theta_{(g)}) ,   \qquad  \cN_{e, \Phi^{-1}} = \bigoplus_{(g)}   \sum_{\theta_{(g)}}  (1-\theta_{(g)} )\cN  (\theta_{(g)}) ,
  \na
 as  linear combinations
  of vector bundles with rational coefficients, or  as    elements in 
  \[
  K_{orb}^0(\tilde \XX) \otimes \Q  = \bigoplus_{(g)}  K_{orb}^0(\XX_{(g)}) \otimes \Q.
  \]

    \begin{theorem} \label{main:2}   Given an  almost complex orbifold $\XX$, 
   let $\cN$ be  the normal bundle  of $\tilde\XX^{[2]}$ in $\XX$. Then 
 the obstruction bundle $E^{[2]}$ in  \cite{CR}  satisfies  the 
 following identity 
 \[
 E^{[2]} \oplus  \cN = e_1^* \cN_{e, \Phi}   + e_2^* \cN_{e, \Phi}   + e_{12}^* \cN_{e, \Phi^{-1}} 
 \]
 in $K^0(\tilde \XX^{[2]})\otimes \Q = \bigoplus_{(g_1, g_2)}  K_{orb}^0(\XX_{(g_1, g_2)}) \otimes \Q $.  
   \end{theorem}
  
 Theorem \ref{main:2} implies that  the  linear combination of vector bundle with rational coefficients
 \[
 e_1^* \cN_{e, \Phi}   + e_2^* \cN_{e, \Phi}   + e_{12}^* \cN_{e, \Phi^{-1}}  - \cN, 
 \]
after combining like terms,  is  a genuine vector bundle, which can be identified with the  obstruction bundle $E^{[2]}$. This theorem was  obtained in \cite{ChenHu} for abelian orbifolds and in  \cite{JKK} for smooth Deligne-Mumford stacks.     In this paper, we give a direct  proof of this theorem using an equivariant version of Kawasaki's orbifold index theorem.
We also   employ this theorem  to give an intrinsic  definition of
Chen-Ruan cohomology as Chen and Hu did for abelian orbifolds \cite{ChenHu}.    
    
  The modified delocalized 
  Chern character is given by 
  \[
\widetilde{ch}_{deloc} = \T (\cN_e, \Phi)  \wedge   ch_{deloc}:  K^*_{orb}(\XX)   \longrightarrow    \bigoplus_{(g)}  H^*( \XX_{(g)}, \C)
  \]
where on each component $\XX_{(g)}$,  $\T(\cN_e, \Phi)$ is defined by
\[
 \prod _{\theta_{(g)}}  \T(\cN (\theta_{(g)} ))^{ \theta_{(g)}}  \in H^{* }  (\XX_{(g)} ), 
 \]
associated to  the eigen-bundle decomposition (\ref{cN_g:decomp}).   Here
$\T(V )^m$ is the multiplicative  characteristic class of an orbifold complex vector bundle $V$  corresponding  to the formal power series
$(\dfrac{1-e^x}{x})^m$ for $m \in \Q\cap (0, 1)$.

 \begin{theorem} \label{main:1}
Let $\XX$ be a compact,  almost complex, effective orbifold. Then there is an associative product  $\circ$ 
on the orbifold K-theory  $K^*_{orb}(\XX) \otimes \C$ such that the  modified delocalized  Chern character
\[
\widetilde{ch}_{deloc}:  (  K^*_{orb}(\XX) \otimes \C, \circ )    \longrightarrow  (H^*_{CR} (\XX, \C), *_{CR}), 
\]
is a ring isomorphism. 
\end{theorem}


 In Section 2, we review some basics of orbifold, orbifold K-theory and orbifold index theorem used in this paper.  In Section 3, we give an intrinsic definition of the Chen-Ruan cohomology  after
 we  establish  Theorem \ref{main:2}.   In Section 4,  we
 define the stringy product on  orbifold K-theory of an almost complex compact orbifold, and
 prove Theorem \ref{main:1}.    We also compute a few examples in Section 4, and discover that the stringy product on
the orbifold K-theory of the orbifold  $[G/G]$ (obtained from the conjugation of a finite group on itself)  is different from  the Pontrjagin product on $K_G(G)$. In Section 5, we briefly discuss the stringy product on twisted orbifold K-theory for orbifolds with torsion twistings.      
    
    \section{Review of Orbifolds and Orbifold Index Theory}
 In this section, we will give a brief review of the notion of orbifolds in terms of orbifold atlas and proper
 \'etale groupoids. Then we review  the orbifold index theory using delocalized Chern character. 
 The main references for this section are \cite{ALR},    \cite{Kaw},  \cite{Moe1},  \cite{MoePr} and \cite{Thur}.

 \subsection{Orbifolds and orbifold  groupoids } \ 

  Let $X$ be a  paracompact Hausdorff space.  
An $n$-dimensional orbifold chart for an open subset $U$ of $X$ is a triple $(\tilde U, G, \pi)$ given by a connected open subset  $\tilde U \subset \R^n$,
together  with a smooth   action of a finite group $G$ such that   $\pi:  \tilde U\to    U  $ is the induced quotient map.  An orbifold chart $(\tilde U, G, \pi)$ is called effective if the action of $G$ on $\tilde U$ is effective.  Given   an inclusion  $\iota_{ij}: U_i \hookrightarrow  U_j$, 
an embedding  of orbifold charts
 \[
 (\phi_{ij}, \lambda_{ij}):   (\tilde U_i, G_i, \pi_i)  \hookrightarrow   (\tilde U_j, G_j, \pi_j) 
 \]
 consists of  an injective  group 
 homomorphism  $\lambda_{ij}:  G_i \to G_j$, and an embedding 
$
 \phi_{ij}:    \tilde U_i   \hookrightarrow    \tilde U_j 
$ covering the inclusion $\iota_{ij}$  such that $\phi_{ij}$ is $G_i$-equivariant with respect to  $\phi_{ij}$, that is, 
 \[
 \phi_{ij}(  gx) =       \lambda_{ij}( g)  \phi_{ij}(x ) ,
 \]
 for $x\in \tilde U_i$ and $g\in G_i$.  In the noneffective case, we further require that
the subgroup of $G_i$ acting trivially on $U_i$ is isomomorphically mapped  to the subgroup
of $G_j$ acting trivially on $U_j$. 
  Whenever $U_i\subset U_j \subset U_k$, there exists an element
$g\in G_k$ such that 
\[
g \circ \phi_{ik} =  \phi_{jk} \circ  \phi_{ij}, \qquad 
g \lambda_{ik} g^{-1} = \lambda_{jk}\circ \lambda_{ij}.
\]

\begin{definition}   An {\em orbifold  atlas}  on  $X$  is    a collection of   orbifold charts   $\cU =\{(\tilde U_i, G_i, \pi_i)\}$  for an open covering  $\{U_i\} $  of $X$ such that
\begin{enumerate}
\item  $\{ U_i\}$ is closed under finite intersection.
\item  Given any inclusion $U_i \subset  U_j$, there is an embedding of orbifold charts
$(\phi_{ij}, \lambda_{ij}):   (\tilde U_i, G_i, \pi_i)  \hookrightarrow   (\tilde U_j, G_j, \pi_j)$.
 \end{enumerate}
Two  orbifold  atlases $\cU$ and $\cV$ are equivalent if there is a common orbifold atlas $\cW$  refining  $\cU$ and  $\cV$. 
An orbifold $\XX = (X, \cU)$   is a paracompact Hausdorff space $X$  with an equivalence class of  orbifold  atlases.   Given an orbifold  $\XX = (X, \cU)$ and a point $x\in X$, let $(\tilde U, G, \pi)$ be an orbifold chart around $x$.  Then the local group at $x$ is defined to   be the  
stabilizer of  $\tilde x \in \pi^{-1}(x)$, which is uniquely defined up to conjugation.
\end{definition}

The notion of an  orbifold and many of its invariants can be reformulated using the language of groupoids.   For general orbifolds, the groupoid viewpoint is also essential for  the  $C^*$-algebraic definition of  K-theory and its twisted version in order to get a cohomology theory satisfying the Mayer-Vietoris  axiom.  We briefly recall the definition of groupoids  and their roles in the orbifold theory.

 A Lie groupoid
 $\cG = (\cG_1 \rightrightarrows \cG_0)$  consists of two smooth manifolds $\cG_1$ and $\cG_0$, together with five smooth maps $(s, t, m, u, i)$ satisfying the following properties. 
  \begin{enumerate}
\item  The source map  and the target map $s, t: \cG_1 \to  \cG_0$ are submersions.
\item The composition map
\[
m:  \cG_1 {\ _t}\times_{s \ }\cG_1 =\{(g_1, g_2) \in  \cG_1 \times  \cG_1: t(g_1) = s(g_2)\} \to \cG_1
\]
written as $m(g_1, g_2) = g_1\cdot g_2$, 
satisfies the obvious associative property.
\item The unit map $u: \cG_0 \to \cG_1$ is a two-sided unit for the composition.
\item The inverse map $i: \cG_1 \to \cG_1$, $i(g) = g^{-1}$,  is a two-sided inverse for the composition. 
\end{enumerate}
A Lie groupoid $\cG = (\cG_1 \rightrightarrows \cG_0)$   is proper if $(s, t): \cG_1 \to \cG_0 \times \cG_0$ is proper, and called \'etale if $s$ and $t$ are local diffeomorphisms.

  Let $\cG_1 \rightrightarrows \cG_0$ and 
$\cH_1  \rightrightarrows \cH_0$ be two   Lie  groupoids.  A generalized morphism 
between 
$ \cG$ and  $ \cH  $
  is   a right
principal $\cH$-bundle $P_f$  over $\cG_0$ which is also a left $\cG$-bundle over $\cH_0$ such that the left $\cG$-action and the right $\cH$-action commute,  formally denoted by
\ba\label{H-S}
\xymatrix{
\cG_1 \ar@<.5ex>[d]\ar@<-.5ex>[d]&P_f \ar@{->>}[ld] \ar[rd]&
\cH_1 \ar@<.5ex>[d]\ar@<-.5ex>[d]\\
\cG_0&&\cH_0
}
\na
  For example,  a rank $k$ Hermitian  vector bundle over a   Lie  groupoid   $\cG $ is defined by a $U(k)$-valued cocycle over $\cG$, that is, a
generalized morphism  $ \cG \to U(k)$ where $U(k)$ is viewed as a Lie groupoid with one object. 
Note that generalized morphisms can be composed. This implies that  the pull-back of a  vector  
bundle  over a groupoid by any generalized morphism is well-defined.  Note that a 
generalized morphism $f$
between 
$ \cG$ and  $ \cH  $ is invertible if $P_f$ in (\ref{H-S}) is also a principal $\cG$-bundle
over $\cH_0$. Then $ \cG$ and  $ \cH  $ are called Morita equivalent.

\begin{remark} \begin{enumerate}
\item As observed in \cite{MoePr2} and \cite{LU},  given an orbifold  $\XX = (X, \cU) $, there is a canonical proper \'etale Lie groupoid $\cG[\cU] $, 
locally given by the action groupoid $\tilde U_i \rtimes G_i\rightrightarrows \tilde U_i $.    For  two equivalent  orbifold atlases $\cU$ and $\cV$,      $\cG[\cU] $ and $  \cG[\cV] $ are Morita equivalent.

\item Given a proper \'etale Lie groupoid $\cG$, there is  a canonical orbifold structure on its orbit space 
$|\cG|$, see \cite{MoePr} and Proposition 1.44 in \cite{ALR}. Two Morita equivalent 
proper \'etale Lie groupoids define the same orbifold up to isomorphism (Theorem 1.45 in \cite{ALR}).  
\item  Given  an orbifold $\XX = (X, \cU)$, a proper \'etale Lie groupoid $\cG$ is called a presentation
of $\XX$ if  there is a homeomorphism $f: |\cG| \to X$ such that  $f^*\cU$ agrees with the canonical orbifold structure on $|\cG|$. A proper \'etale Lie groupoid  is also called an orbifold groupoid for simplicity.

\item An orbifold is called  effective  if  any local group acts  effectively on its orbifold chart.  For an $n$-dimensional effective orbifold $\XX = (X, \cU)$, the corresponding proper \'etale Lie groupoid  is Morita equivalent to the action groupoid associated to the $O(n)$-action on the orthonormal frame bundle for  a  Riemannian
metric  on $\XX$.

 \end{enumerate}
\end{remark}

 \begin{definition}  \label{def}  Let $\XX$ be an  orbifold with a  presenting  groupoid $\cG$.    If $|\cG|$ is compact,   the de Rham cohomology     of an   orbifold $\XX = (X, \cU)$,  denoted by 
   $H^*_{orb}(\XX, \R)$, 
 is defined to be the 
   de Rham cohomology   of   $\cG$  which is the cohomology of the $\cG$-invariant de Rham complex 
   $(\Omega^p(\cG), d)$, where 
   \[
   \Omega^p(\cG) =  \{\omega \in \Omega^p (\cG_0) | s^* \omega = t^*\omega\}.
   \]
   If  $|\cG|$ is not compact,    the de Rham cohomology  of $\XX$   is defined to be the de Rham cohomology   with compact supports of  $\cG$, where a  differential form $\omega \in  \Omega^p(\cG) $ has  compact support in $ |\cG|$.
\end{definition}

An orbifold vector bundle $\cE$  over an orbifold $\XX = (X, \cU)   $ is a family  of 
equivariant vector bundles 
\[
\{ (\tilde E_i \to \tilde U_i, G_i )  \}
\]
  such  that 
for any embedding of orbifold charts $\phi_{ij}:   (\tilde U_i, G_i)  \hookrightarrow   (\tilde U_j, G_j)$,  there is   a  $G_i$-equivariant  bundle map   $\tilde \phi_{ij}: \tilde E_i \to \tilde E_j$ covering 
$ \phi_{ij}: \tilde U_i \to \tilde U_j$.   The total space $E=  \bigcup (\tilde E_i/G_i)$ of an orbibundle $\cE \to \XX$  has a canonical orbifold structure  given by $\{(\tilde E_i, G_i)\}$.  
A connection $\nabla$ on an orbibundle $\cE\to \XX$ is a family of invariant connections $\{ \nabla_i\}$ 
 on
$
\{ ( \tilde E_i \to \tilde U_i, G_i)  \}
$
which  are compatible with the  bundle maps $\{\tilde \phi_{ij}\}$. 
 Examples of orbibundles over $\XX$ include its tangent bundle $T\XX$ and its cotangent bundle $T^*\XX$.
In general, orbifold vector bundles over $\XX$ do not define vector bundles over the underlying topological space $X$. 

Recall that a  vector bundle over a Lie groupoid $\cG = (\cG_1 \rightrightarrows \cG_0)$ is a $\cG$-vector bundle $E$  over $\cG_0$, that is, a  vector bundle $E$  with a fiberwise linear action of $\cG$   covering the canonical  action of $\cG$ on $\cG_0$. 
 One can check that an orbibundle
 $\cE = (E, \cU_E)$ defines a canonical vector bundle over the groupoid $\cG[\cU]$.    

  \begin{definition}  \label{def2}  Let $\XX$ be an  orbifold with a  presenting  groupoid $\cG$.   
 If $|\cG|$ is compact,   the orbifold K-theory of $\XX$, denoted by $K^0_{orb}(\XX)$,   is defined to be the Grothendieck ring of isomorphism classes of vector bundles over $\cG$.   If  $|\cG|$ is not compact,    the orbifold K-theory of $\XX$ is defined to be the Grothendieck ring of isomorphism classes of complex  vector bundles  with compact supports over $\cG$. Here a vector bundle   with compact support over $\cG$ is a $\Z_2$-graded  $\cG$-equivariant complex  vector bundle $E = E^0\oplus E^1$  with a $\cG$-equivariant bundle morphism $\sigma:
   E_0 \to  E_1$ such that  the support of $\sigma$
   \[
   \{x\in \cG_0 |   \sigma_x:  E^0_x \to E^1_x \text{ is not an isomorphism}\}
   \]
 defines  a compact set in $|\cG|$. 
\end{definition}

The   de Rham cohomology  and the orbifold K-theory are  well-defined,  as a Morita equivalence between 
two groupoids induces an   isomorphism on de Rham cohomology and orbifold K-theory respectively. The Satake-de Rham theorem for  an orbifold  $\XX = (X, \cU)$ leads to an isomorphism
\[
H^*_{orb}(\XX, \R) \cong H^*(X, \R)
\]
between the  de Rham cohomology and the singular cohomology of the underlying topological space. 
The standard Chern-Weil construction applied to a compact orbifold $\XX$ gives rise to the Chern character map 
\[
ch:  K_{orb}^0(\XX) \longrightarrow H^{ev}_{orb}(\XX, \C),
\]
which is  a ring homomorphism. 
Here the ring structure on $K_{orb}^0(\XX)$ is given  by the tensor product of vector bundles 
and  the ring structure on $H^{ev}_{orb}(\XX, \C)$ is given by the wedge product of differential forms. This Chern character over the complex coefficients  is  not  an isomorphism. 

\begin{remark} \label{C^*}
The definition of orbifold K-theory  can be extended in the usual way to a  $\Z_2$-graded cohomology theory, see \cite{ALR}, 
\[
K^*_{orb}(\XX) =  K_{orb}^0(\XX) \oplus  K_{orb}^1(\XX). 
\]
For an orbifold   as  a quotient of  a  compact Lie group  action  on locally compact manifolds with finite stabilizers,  the orbifold K-theory has  the usual Bott
periodicity, the Mayer-Vietoris exact sequence  and the Thom isomorphism  for  orbifold
 $Spin^c$ vector bundles.  

We can define the orbifold K-theory of an orbifold  $\XX$ as the K-theory of the reduced $C^*$-algebra   $C^*_{red}(\XX)$  of the   canonical proper \'etale groupoid (see Chapter 2 in \cite{Con}).  The    $C^*$-algebraic  orbifold  K-theory      is  a module over  the orbifold  K-theory  using orbifold vector bundles.
 If $\XX$ is presentable,  that is, $\XX$ is the orbifold obtained from a locally free action of a compact Lie group $G$ on a smooth manifold $M$, then $C^*(\XX)$ is Morita equivalent to  the cross-product $C^*$-algebra 
$C(M)\rtimes G$.  This Morita equivalence can be used to define   an   isomorphism  between
 the orbifold K-theory $K_{orb}^*(\XX)$ defined in Definition \ref{def}  and  the K-theory of the reduced $C^*$-algebra   $C^*_{red}(\XX)$  
\[
K_{orb}^*(\XX) \cong K^*_G(M) \cong K_*(C(M)\rtimes G) \cong K_*(C^*_{red}(\XX)).
\]
      Conjecturally, the isomorphism 
 \[
 K_{orb}^*(\XX) \cong K_*(C^*_{red}(\XX))
 \]
 holds  for general orbifolds. 
  \end{remark}

\subsection{Delocalized Chern character and the orbifold index theory} 

For any  compact orbifold $\XX$, there is a  delocalized Chern character, 
\[
ch_{deloc}: K^*_{orb}(\XX) \longrightarrow H^{*} (\tilde \XX, \C)
\]
from the orbifold K-theory of $\XX$ to the 2-periodic   de Rham cohomology of its inertia orbifold $\tilde\XX$. It was defined in \cite{BC2} for proper actions of discrete groups and in \cite{BryNis} for \'etale groupoids.  The
delocalized Chern character made its first   appearance  in the Lefschetz formulas of Atiyah-Bott in \cite{AB}
and the Kawasaki orbifold index theorem  in \cite{Kaw}.   

  Let $\XX  = (X, \cU) $ be an orbifold. Then the set of pairs
\[
\tilde X =\{ (x, (g)_{G_x})|  x\in X, g\in G_x\},
\]
where $(g)_{G_x}$ is the conjugacy class of $g$ in the local group $G_x$, has a natural 
orbifold structure given by
\[
\{ \bigl(\tilde U^{g}, Z_G(g),  \tilde\pi,  \tilde U^{g}/C(g) \big) |  g\in G\}. 
\]
Here for each orbifold chart $(\tilde U,  G,  \pi, U) \in \cU$,  $Z_G(g)$ is the centralizer of $g$ in $G$ and 
$\tilde U^{g}$ is the fixed-point set of $g$ in $\tilde U$.  This orbifold, denoted by $\tilde \XX$, is called the inertia orbifold of $\XX$. The inertia orbifold $\tilde \XX$  consists of a  disjoint union of sub-orbifolds of $\XX$. 

To describe the connected components of $\tilde \XX$, we need to introduce an equivalence relation on the
set of conjugacy classes in local groups as in \cite{CR}.  For each $x\in X$, let 
$(\tilde U_x,  G_x,  \pi_x, U_x)$ be a local orbifold chart at $x$. If $y\in U_x$, then up to conjugation, there is an injective homomorphism of local groups $G_y\to G_x$, hence 
the conjugacy class $(g)_{G_x}$ is well-defined for $g\in G_y$. We define the equivalence to be generated
by the relation $(g)_{G_y}\sim (g)_{G_x}$. Let $\cT_1$ be the set of equivalence classes. Then 
\[
\tilde \XX = \bigsqcup_{(g) \in \cT_1} \XX_{(g)},
\]
where $\XX_{(g)} =\{(x, (g')_{G_x}| g'\in G_x, (g')_{G_x} \sim  (g)\}$.
 Note that
$\XX_{(1)} =\XX$ is called the non-twisted sector and  $\XX_{(g)}$ for $g\neq 1$ is called a twisted sector
of $\XX$.

Let $\cG$ be a proper \'etale Lie groupoid  representing a  compact orbifold 
$\XX =(X, \cU)$.  Then   the  groupoid    associated to the inertia orbifold $\tilde \XX$ is given by
 $\tilde \cG = ( (s, t): \tilde \cG_1 \rightrightarrows \tilde \cG_0)$, where   
\[  \tilde\cG_0 = \{g\in \cG_1| s(g) =t(g)\} ,  \qquad 
  \tilde\cG_1 = \{ (g,  h) \in \cG_1 \times \cG_1  | g  \in\tilde  \cG_0,  t(g) = s(h)\},  \]  with the source map $s(g,    h) =g $ and the target map $t(g , h) = h^{-1} g  h  $.  There is an obvious evaluation map $e: \tilde \cG \to \cG$ which corresponds to the obvious orbifold immersion 
\[
e=\bigsqcup e_{(g)}: \tilde \XX = \bigsqcup_{(g) \in \cT_1} \XX_{(g)} \to \XX.
\]

  Given a complex orbifold bundle $E$ over $\XX$, or a complex vector bundle over its presenting groupoid
$\cG$,   the pull-back bundle $e^*E$ over $\tilde \XX$ or $\tilde \cG$ has a canonical  automorphism $\Phi$.
With respect to a $\cG$-invariant Hermitian metric on $E$,  there is an eigen-bundle decomposition
of $e^*E$
\[
e^* E = \bigoplus_{\theta \in \Q\cap [0, 1) } E_\theta
\]
where  $E_\theta$ is a complex vector bundle over $\tilde \cG$,  on which $\Phi$ acts  by multiplying $e^{2\pi \sqrt{-1} \theta}$.   Define
\[
ch_{deloc} (E) = \sum_\theta e^{2\pi \sqrt{-1} \theta} ch(E_\theta) 
\in   H^{ev} (\tilde \cG, \C)\cong H^{ev}_{orb}(\tilde \XX, \C),
 \]
 where $ch(E_\theta) \in  H^{ev} (\tilde \cG, \C)$ is the ordinary Chern character of  $E_\theta$.

 The odd delocalized Chern character
 \[
 ch_{deloc}: K^1_{orb}(\XX) \longrightarrow H^{odd} (\tilde \XX, \C)
\]
can be defined in the usual way.  Using the standard compactly supported condition,
the delocalized Chern character can be defined for non-compact orbifolds.

 \begin{remark} For a compact quotient orbifold $\XX=[M/G]$ where $G$ is a compact Lie group  acting  on a compact manifold $M$ with finite stablizers, Adem and Ruan obtained a decomposition (Cf. Theorem
 5.1 and Corollary 5.2 in \cite{AR})
 \[
 K^*_{orb}(\XX)\otimes  \C \cong  K^*_G(M) \otimes \C \cong \bigoplus_{\{\text{conjugacy class}\  \<g\> : g\in G\}}
 K^*(M^g/Z_G(g)) \otimes \C,
 \]
 where the closed submanifold  $M^g$ is the fixed  point set of the $g$-action, and $Z_G(g)$ is
 the centralizer of $g$ in $G$. 
 Applying the ordinary Chern character on each $K^*(M^g/Z_G(g))$, we get an
alternative definition of  the delocalized  Chern character over $\C$:
\ba\label{AR:deloc}
ch_{deloc}:   K^*_G(M) \otimes  \C \longrightarrow  \bigoplus_{\{  \<g\> : g\in G\}}
 H^*(M^g/Z_G(g),  \C).
 \na
 Going through the proof of Theorem 5.1 in \cite{AR}, particularly  the ring map from the representation ring of the cyclic subgroup $\<g\>$ generated by $g$ to the the cyclotomic field $\Q (e^{2\pi\sqrt{-1}/|(g)|})$, one  can check that the delocalized  Chern character (\ref{AR:deloc}) agrees with the delocalized Chern
 character defined by eigen-bundle decompositions. Here we identify the inertia orbifold $\tilde\XX$ with
 the orbifold $\bigsqcup_{\<g\>} [M^g/Z_G(g)]$. 
 \end{remark}

  \begin{proposition}  \label{deloc:iso}
  For any compact presentable  orbifold $\XX$, the delocalized Chern character 
  gives a ring isomorphism 
  \[
  ch_{deloc}: K^*_{orb}(\XX) \otimes_\Z \C  \longrightarrow H^{*} (\tilde \XX, \C)
  \]
  over $\C$. 
  \end{proposition} 
  
  \begin{proof} The proof follows from the isomorphism  for orbifolds obtained from
  a finite group action on a locally compact manifold  and the  Mayer-Vietoris sequence  for open covers.  Recall that the canonical
  groupoid $\cG$ associated to an orbifold chart $\{(\tilde U_i, G_i, \pi_i)\}$, when restricted to
  each open set $\tilde U_i$,  is an action groupoid $\tilde U_i\rtimes G_i$. With compactly supported
  K-theory and de Rham cohomology, one has the following isomorphism (Theorem 1.19
  in \cite{BC2})  of vector spaces over $\C$
  \[
  ch_{deloc}: K^*_{G_i} (\tilde U_i  ) \otimes_\Z \C  \longrightarrow H^{*}_{c}(\tilde U_i,  G_i)
  \] 
  where $  H^{*}_{c}(\tilde U_i,  G_i) =\bigl[  \oplus_{g\in G_i} H_c^*(\tilde U_i^g, \C)\bigr]^{G_i}$
  with $\tilde U_i^g$  the fixed point submanifold  of the $g$-action.  From the definition, we see that $H^{*}_{c}(\tilde U_i,  G_i)$ is the de Rham cohomology of the inertia groupoid of 
  the action groupoid $\tilde U_i\rtimes G_i$.  By an induction argument, we can apply the 
  Mayer-Vietoris sequence for open covers and the five lemma to show that 
  \[
  ch_{deloc}: K^*_{orb}(\XX) \otimes_\Z \C  \longrightarrow H^{*} (\tilde \XX, \C)
  \]
  is an  isomorphism of vector spaces.  For two vector bundles $E$ and $F$ with eigen-bundle
  decompositions
  \[
  e^* E = \bigoplus_{\theta } E_\theta \qquad  e^* F = \bigoplus_{\zeta  } F_\zeta,
\]
we have
\[
e^*(E\otimes F) = \bigoplus_{\theta, \zeta}    E_\theta \otimes  F_\zeta,
\]
from which   one immediately gets 
\[
ch_{deloc} (E\otimes F)  = ch_{deloc} (E)\cup ch_{deloc} (F).
\]
Hence,  $ch_{deloc}: K^0_{orb}(\XX) \otimes_\Z \C  \longrightarrow H^{ev}_{orb}(\tilde \XX, \C)  $ is a  ring isomorphism.  The ring isomorphism for the odd delocalized Chern character can be proved by the standard de-suspension argument.
  \end{proof}
  
  \begin{remark}\label{equ:Kar}  The delocalized Chern character can be applied to write the Kawasaki orbifold index 
  \cite{Kaw} as follows. Let $\XX$ be a compact  almost complex orbifold with a Hermitian connection on  the
tangent bundle $T\tilde \XX$ of   the inertia orbifold $\tilde\XX$.  Let $E$ be a complex orbifold Hermitian vector bundle with a Hermitian
  connection. Let $\Dirac^\pm_E$ be the corresponding $Spin^c$ Dirac operator.   The orbifold index formula in \cite{Kaw} can be expressed as
  \[
  Index(\Dirac^+_E) = \int_{\tilde\XX}^{orb} ch_{deloc}(E) Td_{deloc}(\XX) 
  \]
  where $Td_{deloc}(\tilde\XX)  \in H^*(\tilde \XX)$  is the delocalized Todd class of   $\XX$, whose  $\XX_{(g)}$-component is given by
  \[
 \dfrac{ Td (\XX_{(g)}) }{det \big(1-g^{-1} e^{F_{\cN_{(g)}}/2\pi i}  \big)}.
 \]
 Here $Td (\XX_{(g)})$ is the Todd form of $\XX_{(g)}$ and $F_{\cN_{(g)}}$ is the 
 curvature of the induced Hermitian connection on the normal bundle $\cN_{(g)}$.  In this paper, we will use the equivariant orbifold index formula for a finite group $H$ acting trivially on $\XX$.  Here we briefly discuss this version of the  equivariant  orbifold index formula. When $H$ acts trivially on $\XX$, we have 
 \ba\label{iso:H}
 K_{orb, H}^0(\XX) \cong K_{orb}^0(\XX) \otimes  R(H)
 \na
 where $R(H)$ is the representation ring of $H$. Composing this isomorphism with the delocalized 
 Chern character, we get 
 \[
 ch_{deloc, H}:  K_{orb, H}^0(\XX) \longrightarrow H^{ev} (\tilde \XX, \C) \otimes R(H).
 \]
 Then the equivariant orbifold index formula can be written as 
 \[
 Index_H(\Dirac^+) = \int_{\tilde\XX}^{orb} ch_{deloc, H}(E) Td_{deloc}(\XX).
  \]
  In particular, for each $h\in H$ and $E= \oplus_i E_i \otimes V_i$ under the isomorphism (\ref{iso:H}), we have
  \[
  Tr(h|_{Ker \Dirac^+}) -  Tr(h|_{Ker \Dirac^-}) =  \sum_i  Index(\Dirac^+_{E_i}) Tr(g|_{V_i}).
  \]
     \end{remark}

\section{Intrinsic description of Chen-Ruan cohomology}  

For an  almost complex  abelian orbifold $\XX$, Chen and Hu gave a classical definition of the
Chen-Ruan product using an intrinsic definition of Chen-Ruan's obstruction bundle, see
the proof of Proposition 1 in section 3.4 of \cite{ChenHu}.  For a smooth Deligne-Mumford stack, a similar description of  Chen-Ruan's obstruction bundle was obtained by Jarvis, Kaufmann and Kimura in \cite{JKK}.    In this section, we give an 
intrinsic definition of Chen-Ruan's obstruction bundle for  an almost complex orbifold  $\XX$ with associated 
 Lie groupoid $\cG$. 

 Let $\tilde\XX  =\bigsqcup_{(g) \in \cT_1} \XX_{(g)}$ be the inertia orbifold with  associated inertia groupoid   $\tilde\cG$ and 
 the evaluation map $e: \tilde \cG\to \cG$.   
The  $k$-sector  $\tilde\XX^{[k]}$ of $\XX$ is defined to be the orbifold  on the
set of all pairs 
\[
  (x, (g_1, \cdots, g_k)_{G_x}),
  \]
where $(g_1, \cdots, g_k)_{G_x}$ denotes the conjugacy  class of k-tuples.  Here two 
k-tuples $(g_1^{(i)}, \cdots, g_k^{(i)})_{G_x}$, $i=1, 2$, are conjugate if there is a $g\in G_x$ such that 
$g_j^{(2)} = g g_j^{(1)}g^{-1}$ for all $j=1, \cdots, k$.  The  $k$-sector orbifold $\tilde\XX^{[k]}$ 
consists of a disjoint union of sub-orbifolds of $\XX$
\[
\tilde \XX^{[k]} = \bigsqcup_{({\bf g}) \in \cT_k} \XX_{({\bf g})},
\]
where $\cT_k$ denotes the set of  equivalence classes of  conjugacy   k-tuples in local groups.
Then  the groupoid $\tilde\cG^{[k]}$ associated to the k-sector  $\tilde\XX^{[k]}$  for $k\geq 2$  is given by 
  \[
  \tilde\cG^{[k]}_0 = \{(g_1, g_2, \cdots, g_k) \in  \tilde\cG_0 \times_e \tilde\cG_0 \cdots  \times_e \tilde\cG_0\},
  \]
  and 
  \[
   \tilde\cG^{[k]}_1 = \{(g_1, g_2, \cdots, g_k, h)  \in   (\tilde\cG^{[k]}_0,
  \cG_1)|   t (g_i) = s(h)\},
  \]
  with the source map $s(g_1, g_2, \cdots, g_k, h) = (g_1, g_2, \cdots, g_k)$ and 
  the target map $t(g_1, g_2, \cdots, g_k, h) = (h^{-1}g_1 h, h^{-1} g_2 h , \cdots, h^{-1}  g_k  h)$.

As in \cite{CR}, $\tilde \cG^{[k]}$ can be identified with the orbifold moduli space of constant pseudo-holomorphic maps from an orbifold sphere with $k+1$ orbifold points to $\XX$.  The obstruction  bundle over  the orbifold moduli space, given by the cokernel of the Cauchy-Riemann operator over the orbifold sphere coupled with the pull-back of the tangent bundle $T\cG$ of $\cG$, defines a complex   vector bundle $E^{[k]}$  over   the groupoid of k-sectors $\tilde\cG^{[k]}$, for $k\geq 2$.

\begin{definition}\label{log}
Let $E$ be  any complex  vector bundle   with an automorphism  $\Phi$ of finite order over a proper \'etale  groupoid $\cG$.  Choose a Hermitian metric on $E$ preserved by $\Phi$. Then $E$ has an eigen-bundle
decomposition
\[
E = \bigoplus_{m_i \in \Q\cap [0, 1)}   E (m_i ) 
\]
where $\Phi$ acts on $E(m_i)$ as multiplication by  $e^{2\pi \sqrt{-1}  m_i }$ for $m_i\in\Q \cap  [0, 1)$.  We define  
\[
E_\Phi =  \sum_{m_i \in \Q\cap ( 0, 1)}  m_i   E (m_i ), \qquad \text{and}\qquad 
E_{\Phi^{-1}} =  \sum_{m_i \in \Q\cap (0, 1)}  (1- m_i  )  E (m_i )
\]
as a linear combination of vector bundles with rational coefficients, or an element in 
$K^0(\cG)\otimes \Q$.
  \end{definition}

Given an almost complex orbifold $\XX$,  any orbifold complex vector  bundle over  a compact orbifold $\tilde\XX$,  or equivalently  any  vector bundle over $\tilde \cG$, 
has an automorphism.   Specifically,  if $E$ is  a   vector bundle over $\tilde \cG$, then
  the fiber  $E_g$ at $g\in \tilde\cG_0$ has a linear isomorphism  induced by the action of   
$g$.      
In particular,   $e^*T\cG$ is a complex vector bundle 
over $\tilde \cG$ with an automorphism
$\Phi$,  and $T\tilde\cG$ is a sub-bundle of $e^*T\cG$  on which $\Phi$ acts trivially. Let  $\cN_e$
be the normal bundle of the evaluation map $e$ with the  induced automorphism $\Phi$.
We can choose a $\cG$-invariant Hermitian metric  on $\cG$ such that   $\Phi$ acts unitarily on $\cN_e$.  So  the automorphism preserves the orthogonal  decomposition
\[
 e^*T\cG = T\tilde \cG \oplus \cN_e.
\]
As a complex vector bundle over $\tilde \cG$,  $\cN_e$  admits an eigen-bundle  decomposition 
    \[
  \cN_e  =  \bigsqcup_{(g)}  \cN_{(g)}  =     \bigsqcup_{(g)}   \bigoplus_{\theta_{(g)}   } \cN (\theta_{(g)})  \]
  where $\Phi$  acts on $\cN  (\theta_{(g)})   $ as   multiplication by   $e^{2\pi \sqrt{-1}\theta_{(g)} }$ with $\theta_{(g)} \in (0, 1)$, as $\Phi$ does not have eigenvalue $1$ on $\cN_e$.     

   Consider  the following  commutative diagram 
 \ba\label{e:map}   \xymatrix{
   \tilde\cG^{[2]} \ar[rr]^{e_1}\ar[dd]_{e_2} \ar[dr]^{e_{12} } &&  \tilde\cG \ar[dd]^{e}     \\
   &  \tilde\cG \ar[dr]^e &     \\
    \tilde\cG \ar[rr]^e &&  \cG. 
   }
\na
Let $\cN$ be the normal bundle to the map $e \circ e_{12}  = e \circ e_{1}\  =  e \circ e_{2}:
 \tilde\cG^{[2]} \to \cG$.  Then 
   $\cN$ as a complex vector bundle over $\tilde \cG^{[2]}$,   has three automorphisms 
   \begin{enumerate}
\item $\cN \cong \cN_{e_1} \oplus e_1^*\cN_{e}$ with an automorphism $\Phi_1 = Id\oplus e_1^*\Phi$, here $ \cN_{e_1}$ is the normal bundle to the map $e_1$. Then  $\cN_{\Phi_1} $, by Definition \ref{log}, is given by 
\[
\cN_{\Phi_1} = e_1^* \cN_{e, \Phi}.
\]

\item  $\cN \cong \cN_{e_2} \oplus e_2^*\cN_{e}$ with an automorphism $\Phi_2 = Id\oplus e_2^*\Phi$, here $ \cN_{e_2}$ is the normal bundle to the map $e_2$.  Then we have
\[
\cN_{\Phi_2} = e_2^* \cN_{e, \Phi}.
\]

\item  $\cN \cong \cN_{e_{12} } \oplus e_{12}^*\cN_{e}$ with an automorphism $\Phi_{12} = Id\oplus e_{12}^*\Phi$, here $ \cN_{e_{12}}$ is the normal bundle to the map $e_{12}$. Then we have
\[
\cN_{\Phi_{12}^{-1} } = e_{12}^* \cN_{e, \Phi^{-1}}.
\]

\end{enumerate}
When $\XX$ is compact, all  of these three automorphisms have finite order. 
  
  The main result of this section  is the following intrinsic description of the obstruction bundle
  $E^{[2]}\to \tilde\cG^{[2]}$ in the definition of  the Chen-Ruan cohomology. This  generalises the results  of  \cite{ChenHu} and \cite{JKK} to general almost complex orbifolds.  
   Our proof using the  orbifold index theorem  was inspired by the  proof in \cite{ChenHu} for abelian cases, and  is simpler  than the proof given in \cite{JKK} for smooth Deligne-Mumford stacks. 
  
  \begin{theorem} \label{theorem:2}
   The obstruction bundle $E^{[2]}$  in the construction of the Chen-Ruan product   satisfies    the following identity
  \[
  E^{[2]}  \oplus  \cN = e_1^* \cN_{e, \Phi} + e_2^*\cN_{e,\Phi}  + e_{12}^* \cN_{e, \Phi^{-1}}
  \]
  in $K^0(\tilde\cG^{[2]})\otimes \Q$. 
  \end{theorem}
  \begin{proof}  We first recall the definition of the obstruction bundle from  Chapter 4.3 in \cite{ALR} 
and   Section 4 in  \cite{CR}.  Identify  $\tilde\cG^{[2]}$ as  the moduli space of constant representable orbifold morphisms   from the  orbifold
Riemann sphere  $ S^2$ with three orbifold points to $\XX$.     Given a point 
    \[
    (  g_1, g_2) \in \tilde\cG^{[2]}_0
    \]
    with $s(g_i) = t(g_i) =x$, 
   let $N$ be a finite subgroup generated by $g_1, g_2$ in the local  group 
   \[
   G_x = s^{-1}(x)\cap t^{-1}(x)\]
   By Lemma 4.5 in \cite{ALR}, up to an isomorphism, $N$ depends only on the connected component of  $\tilde\XX^{[2]}$, the orbifold associated to $\tilde\cG^{[2]}$.  Let $m_1$, $m_2$ and $m_3$ be the   order of $g_1$, $g_2$ and $g_1g_2$ respectively. There is a smooth compact  Riemann surface  $\Sigma$ such that
   $\Sigma/N$ is an   orbifold
Riemann sphere  $  (S^2, (m_1, m_2, m_3))  $ with three orbifold points of multiplicities $(m_1, m_2, m_3)$.     Note that
$\Sigma$ is given by the quotient of the orbifold universal cover of $  (S^2, (m_1, m_2, m_3))  $  by the kernel of the surjective
homomorphism  (Cf. (4.14) in \cite{ALR})
\[
\pi_1^{orb} (S^2, (m_1, m_2, m_3))  \to N.
\]

   The constant  orbifold morphism corresponding to $ (x,  g_1, g_2)$
      is represented by 
   an ordinary constant  map 
   \[
   \tilde f_x: \Sigma \to \tilde U_x
   \]
   for an orbifold chart  $(\tilde U_x, G_x)$ at $x$.  The elliptic complex for the
   obstruction bundle over the point $ (x,  g_1, g_2)$ is the $N$-invariant part of 
   the elliptic complex 
   \[
   \bar{\partial}_\Sigma: \Omega^0(\Sigma, T_x\cG) \longrightarrow \Omega^{0,1}(\Sigma, T_x\cG).
   \]
   So the  tangent space of the moduli space $\tilde\cG^{[2]} $ at    $ (  g_1, g_2) \in \tilde\cG^{[2]}_0$
   is given by
   \ba\label{tangent}
   T_{(g_1, g_2)}\tilde\cG^{[2]}=    \big(H^{0}(\Sigma) \otimes T_x\cG  \big)^N 
   \na
   and the   obstruction bundle over the point $ (  g_1, g_2) \in \tilde\cG^{[2]}_0$ is given by
    \[
    E^{[2]}_{ (  g_1, g_2)} = \big (H^{0, 1}(\Sigma) \otimes  T_x \cG \big)^N ,   \]
   as $Z_{G_x}(g_1)\cap Z_{G_x}(g_2)$-representations.  Here $Z_{G_x}(g_1)$ and $Z_{G_x}(g_2)$ are the centralizers of $g_1$ and $g_2$ respectively in the local  group $G_x$.  
   
   Note that  $(e\circ e_{12})^* T_x\cG = \cN_{(g_1, g_2)} \oplus T_{(g_1, g_2)}\tilde\cG^{[2]}$ and $N$ acts on $T_{(g_1, g_2)}\tilde\cG^{[2]}$ trivially.  We have
   \[\begin{array}{lll}
 &&  \big(H^{0}(\Sigma) \otimes T_{(  g_1, g_2)} \cG  \big)^N \\
    & =  &    \big(H^{0}(\Sigma) \otimes  T_{(g_1, g_2)}\tilde\cG^{[2]}  \big)^N \oplus  \big (H^{0}(\Sigma) \otimes   \cN_{(g_1, g_2)}  \big)^N\\
    &=& T_{(g_1, g_2)}\tilde\cG^{[2]} \oplus  \big (H^{0}(\Sigma) \otimes   \cN_{(g_1, g_2)}  \big)^N. 
\end{array}   \]
   Combining  with (\ref{tangent}), this leads to 
 \ba\label{zero}
     \big (H^{0}(\Sigma) \otimes   \cN_{(g_1, g_2)}  \big)^N =0.
       \na

   Applying the orbifold  index formula (Proposition 4.2.2 of \cite{CR})  to  the $N$-invariant part of 
   the elliptic complex 
   \[
   \bar{\partial}_\Sigma|_{T_{(g_1, g_2)}\tilde\cG^{[2]}}: \Omega^0(\Sigma, T_{(g_1, g_2)}\tilde\cG^{[2]}) \longrightarrow \Omega^{0,1}(\Sigma, T_{(g_1, g_2)}\tilde\cG^{[2]}),
   \]  
  we get
 \ba\label{index}
 Index ( \bar{\partial}_\Sigma|_{T_{(g_1, g_2)}\tilde\cG^{[2]}} ) = 
  \dim T_{(g_1, g_2)}\tilde\cG^{[2]} - 0. 
\na
Here  the contributions  to the orbifold  index formula  from the three singular points are all zero due to the trivial action of $N$ on $T_{(g_1, g_2)}\tilde\cG^{[2]}$. 
As 
\[ \begin{array}{lll}
&&  Index ( \bar{\partial}_\Sigma|_{T_{(g_1, g_2)}\tilde\cG^{[2]}} )  \\
&=&   
 \dim  \big(H^{0}(\Sigma) \otimes  T_{(g_1, g_2)}\tilde\cG^{[2]}  \big)^N - \dim 
 \big (H^{0, 1}(\Sigma) \otimes  T_{(g_1, g_2)}\tilde\cG^{[2]}  \big)^N \\
 &=&  \dim T_{(g_1, g_2)}\tilde\cG^{[2]} -  \dim 
 \big (H^{0, 1}(\Sigma) \otimes  T_{(g_1, g_2)}\tilde\cG^{[2]}  \big)^N ,
 \end{array}
  \]
  together with (\ref{index}), 
 we obtain    $   \big (H^{0, 1}(\Sigma) \otimes  T_{(g_1, g_2)}\tilde\cG^{[2]}  \big)^N =0$.  Therefore, 
  \[\begin{array}{lll}
  E^{[2]}_{ (  g_1, g_2)}  
&=& \big (H^{0, 1}(\Sigma) \otimes  T_{(g_1, g_2)}\tilde\cG^{[2]}   \big)^N \oplus  \big (H^{0, 1}(\Sigma) \otimes  \cN_{(g_1, g_2)}  \big)^N\\[3mm]
&=&  \big (H^{0, 1}(\Sigma) \otimes  \cN_{(g_1, g_2)}  \big)^N
\end{array}
\]
 as    $Z_{G_x}(g_1)\cap Z_{G_x}(g_2)$-representations,    
   which can be glued together using the  groupoid action to form a complex vector bundle  $E^{[2]}$ over the groupoid $\tilde\cG^{[2]}$. 
 
Let 
\ba\label{decomp}
\cN_{(g_1, g_2)} = \bigoplus_\lambda V_\lambda
\na
be the decomposition into   irreducible   $Z_{G_x}(g_1)\cap Z_{G_x}(g_2)$-representations.
 
Note that the automorphisms $\Phi_1$, $\Phi_2$ and $\Phi_{12}^{-1}$, induced by the action of
$g_1$, $g_2$ and $(g_1g_2)^{-1}$ respectively,  preserve each irreducible component 
$V_\lambda$ in  (\ref{decomp}), and commute with the action of $N$.  Any eigen-subspace of $V_\lambda$ of $\Phi_1$, $\Phi_2$ or  $\Phi_{12}^{-1}$ 
is also a representation of $Z_{G_x}(g_1)\cap Z_{G_x}(g_2)$. The irreducibility of $V_\lambda$
implies that 
 \begin{enumerate}
\item $ (V_\lambda)_{\Phi_1} = \theta^\lambda_1 V_\lambda $ where $\Phi_1$ acts on 
$V_\lambda $ as multiplication by $e^{2\pi\sqrt{-1} \theta^\lambda_1}$ for $\theta_1 \in (0, 1)\cap \Q$,
\item $ (V_\lambda)_{\Phi_2} = \theta^\lambda_2 V_\lambda $ where $\Phi_2$ acts on 
$V_\lambda $ as multiplication by  $e^{2\pi\sqrt{-1} \theta^\lambda_2}$ for $\theta_2 \in (0, 1)\cap \Q$,
\item $ (V_\lambda)_{\Phi^{-1}_{12}} = \theta^\lambda_{12} V_\lambda $ where $\Phi^{-1}_{12}$ acts on 
$V_\lambda $ as multiplication by  $e^{2\pi\sqrt{-1} \theta^\lambda_{12}}$ for $\theta_{12} \in (0, 1)\cap \Q$.
\end{enumerate}

Note that  $  Z_{G_x}(g_1)\cap   Z_{G_x}(g_2)$ acts trivially the orbifold sphere $(S^2, (m_1, m_2, m_3))$. 
Applying the equivariant  version of Kawasaki's  orbifold  index formula (See Remark \ref{equ:Kar}) to the $N$-invariant part of
the elliptic complex
\[
 \bar{\partial}_\Sigma: \Omega^0(\Sigma, V_\lambda) \longrightarrow \Omega^{0,1}(\Sigma, V_\lambda),
   \]
we get, for $g\in Z_{G_x}(g_1)\cap Z_{G_x}(g_2)$, 
\ba\label{good:1}\begin{array}{lll}
&&  Tr \big(g|_{(H^{0}(\Sigma) \otimes  V_\lambda)^N }\big)  - Tr \big( g|_{(H^{0, 1}(\Sigma) \otimes  V_\lambda )^N}  \big)  \\[2mm]
&=& Tr (g|_{V_\lambda})  -   (\theta^\lambda_1 + \theta^\lambda_2 + \theta^\lambda_{12}) Tr (g|_{V_\lambda}).
\end{array}
\na
 From (\ref{zero}), we get  $   \big (H^{0}(\Sigma) \otimes   V_\lambda \big)^N =0$ for each irreducible representation $V_\lambda$ in (\ref{decomp}).  Then (\ref{good:1})  gives rise to 
\[
Tr \big( g|_{(H^{0, 1}(\Sigma) \otimes  V_\lambda )^N}  \big)+   Tr (g|_{V_\lambda}) =   (\theta^\lambda_1 + \theta^\lambda_2 + \theta^\lambda_{12}) Tr (g|_{V_\lambda}),
\]
for each $g \in Z_{G_x}(g_1)\cap Z_{G_x}(g_2)$. 
This implies that,  as a  representation of
 $ Z_{G_x}(g_1)\cap Z_{G_x}(g_2)$,  
 \ba\label{good:2}
  \big (H^{0, 1}(\Sigma) \otimes  V_\lambda \big)^N   \oplus V_\lambda = (V_\lambda)_{\Phi_1} 
  + (V_\lambda)_{\Phi_2} +  (V_\lambda)_{\Phi^{-1}_{12}}, 
  \na
  after combining  like terms.
 Taking the direct sum over $\lambda$ as in (\ref{decomp}), we obtain
  \ba\label{fiberwise}
  E^{[2]}_{(g_1, g_2)}  \oplus \cN = (e_1^*\cN_{e, \Phi} )_{(g_1, g_2)}
  +( e_2^*\cN_{e, \Phi})_{(g_1, g_2)} +  (e_{12}^*\cN_{e, \Phi^{-1}})_{(g_1, g_2)}
  \na
  in the  representation ring of $Z_{G_x}(g_1)\cap Z_{G_x}(g_2)$ over the rational  coefficients $\Q$.  
  
  The normal bundle $\cN$ is a complex vector bundle over the groupoid $\tilde\cG^{[2]}$ so  the fiberwise identity (\ref{fiberwise}) leads to the identity
   \[
  E^{[2]}  \oplus  \cN = e_1^* \cN_{e, \Phi} + e_2^*\cN_{e,\Phi}  + e_{12}^* \cN_{e, \Phi^{-1}}
  \]
  in $K^0(\tilde\cG^{[2]})\otimes \Q$.  This completes the proof of the theorem.
   \end{proof}

  \begin{remark} \label{re:main:2} \begin{enumerate}
\item  One particular consequence of Theorem \ref{theorem:2} is that the  linear combination of vector bundles with rational coefficients 
\[
  e_1^* \cN_{e, \Phi} + e_2^*\cN_{e,\Phi}  + e_{12}^* \cN_{e, \Phi^{-1}}- \cN,
\]
after combining   like terms, is a complex vector bundle which can be identified with the obstruction bundle $E^{[2]}$  in the construction of the Chen-Ruan product.
\item In \cite{Hep}, Hepworth also provides another  description of the obstruction bundle.  As remarked in the introduction of \cite{Hep}, his description of the obstruction bundle involves the theory of orbifold Riemann surfaces to get certain inequalities regarding the degree shifting in the definition of the Chen-Ruan cohomology. 
 \end{enumerate}
 \end{remark}
  
  With this intrinsic description of the obstruction bundle, we can  extend 
  Chen-Hu's alternative definition of  the Chen-Ruan cohomology ring  for abelian orbifolds
  to general orbifolds by identifying
  \[
  H_{CR}^*( \XX) = H^{*-2\iota}(\tilde \XX)
  \]
  with certain formal cohomology classes on $\XX$.  
 We firstly recall the definition of the  Chen-Ruan cohomology of an almost complex, compact orbifold. More details  can be found in \cite{CR}.

  \begin{definition} \label{CR} The Chen-Ruan cohomology $(H^*_{CR}(\XX), *_{CR} )$  of $\XX$ is defined to be  
\[
H^d_{CR}(\XX) =      \bigoplus_{(g)\in \cT_1} H^{d-2\iota_{(g)}} (\XX_{(g)}, \C)
\]
 with a degree shift and  a new  product  $*_{CR}$ 
given  by the following formula, for $\omega_1, \omega_2  \in H^* (\tilde \XX, \C)$, 
\[
  \omega_1 *_{CR}   \omega_2  =( e_{12})_* \bigl(e_1^*\omega_1 \cup e_2^*\omega_2 \cup e(E^{[2]}) \bigr), 
\]
where $e(E^{[2]}) \in H^*(\tilde \XX^{[2]}, \C)$ is the cohomological Euler  class of the obstruction bundle $E^{[2]}$. 
  The   degree shifting number $\iota: \bigsqcup_{(g) \in \cT_1} \XX_{(g)}   \to \Q$ is   defined
to be 
   \[
   \iota_{(g)}  = \sum_{\theta_{(g)}}  rank_\C (\cN (\theta_{(g)}) ) \theta_{(g)} ,
   \]
a locally constant function on $\tilde\XX$.       
\end{definition}

\begin{remark}The Chen-Ruan product can also be defined using the 3-point function as follows. Given $(g_1, g_2) \in \cT_2$, let $e_1, e_2$ and $e_{12}$  be the orbifold embeddings in (\ref{e:map}). Denote  
\[
e_3 = I \circ e_{12}:  \XX_{(g_1, g_2)} \longrightarrow \XX_{(g_3)}
\]
where $g_3 = (g_1 g_2)^{-1}$ and $ I:  \tilde \XX \to \tilde \XX$ is the involution defined by $(x, (g)_{G_x})\mapsto 
(x, (g^{-1})_{G_x})$.  Then   for $\omega_1 \in H^*(\XX_{(g_1) }, \C)$ and 
$\omega_2 \in H^*(\XX_{(g_2) }, \C)$, 
  the Chen-Ruan product 
\[
\omega_1 *_{CR} \omega_2 \in H^*(\XX_{(g_1g_2)}, \C)
\]
is uniquely determined by the 3-point function
\ba\label{3-point}
\< \omega_1 *_{CR} \omega_2 , \omega_3 \> = \int_{\XX_{(g_1, g_2)}} ^{orb} e_1^* \omega_1
\wedge e_2^* \omega_2 \wedge e_3^* \omega_3 \wedge e(E^{[2]}_{(g_1, g_1)} ) 
\na
for any  compactly supported $\omega_3 \in H^*(\XX_{(g_3)}, \C)$. 
\end{remark}

 Given an oriented real   orbifold   vector bundle $V$ over $\XX$, or equivalently, an oriented   vector bundle $V$ over $\cG$, there is a compactly supported differential  form  $\Theta_V$  of $V$ such that
 \[
 \int^{orb}_{V} \alpha \wedge \Theta_V = \int^{orb}_{\XX} i^*\alpha,
 \]
 for any differential form $\alpha$ on $V$, and $i: \XX \to V$ is the  inclusion map defined by the  zero section.  Such a differential form defines the Thom class $\Theta(V)$ of $V$. 
 
 Denote by
 \[
 \Theta(\cN_e)  = \big(  \cdots,  \Theta(\cN_{e_{(g)} } ), \cdots      \big) \in \bigoplus_{(g) \in \cT_1}  H^*(\cN_{(g)})
 \]
  the Thom class of the normal bundle  $\cN_e$ over $\tilde \XX =\bigsqcup_{(g)\in \cT_1}  \XX_{(g)}$ for the evaluation map $e:  \tilde \XX \to \XX$.   Note that the degree of 
  $\Theta(\cN_{(g)})$   is $2n_{(g)}$, where $n_{(g)}$ is  the complex  codimension of $\XX_{(g)}$ in $\XX$. 
  The homomorphism 
  \ba\label{Thom}
  H^*(\tilde \XX)  =  \bigoplus_{(g) \in \cT_1}  H^*(\XX_{ (g)} )
   \longrightarrow   \bigoplus_{(g) \in \cT_1}  H^{*+2n_{(g)} } (\cN_{(g)})
  \na
   sending $\omega_{(g)}  \in H^*(\tilde \XX_{(g)} )$ to $\pi_{(g)}^*\omega_{(g)}   \wedge  \Theta(\cN_{(g)})$,  is the Thom isomorphism. 
   Here $\pi_{(g)}:   \cN_{e_{(g)} }  \to  \XX_{(g)}$ is the bundle projection.     Identifying a neighbourhood of the zero section in $\cN_{e_(g)}$  containing the support of   $\Theta(\cN_{e_(g)})$,  with the orbifold
  neighbourhood  of $\tilde \XX_{(g)}$ in $\XX$, we obtain the push-forward map
 \[
 e_*:  H^*(\tilde \XX)  = \bigoplus_{(g) \in \cT_1}  H^*(\XX_{ (g)} )  \longrightarrow   
  \bigoplus_{(g) \in \cT_1}    H^{* + 2n_{(g)} } (\XX).  \]
  
  As before,  let 
\[
\cN_e   =     \bigsqcup_{(g)\in \cT_1}   \bigoplus_{\theta_{(g)}   } \cN  (\theta_{(g)})  \longrightarrow
 \bigsqcup_{(g)}  \XX_{(g)} \]
be  the eigen-bundle decomposition for the automorphism $\Phi$. 
   
 \begin{definition}
 The fractional Thom class of the normal bundle $( \cN_{(g)}, \Phi)$  over 
 $ \XX_{(g)}$ is defined by the  formal wedge  product
 \[
 \Theta(\cN_{(g)}, \Phi) =  \bigwedge_{\theta_{(g)} } \big(\Theta (\cN (\theta_{(g)} )\big)^{\theta_{(g)}},
 \]
 of  formal degree  $2\iota_{(g)}  = \sum_{\theta_{(g)}} \theta_{(g)} $, with    compact support  in   a neighborhood of  $ \XX_{(g)}$ in $\XX$. 
  \end{definition}

 For each    $ (g)\in \cT_1$, denote by $  H_\Phi^{*+2\iota_{(g)} }(\XX)$ the set of formal  wedge  products 
\[
 H_\Phi^{*+2\iota_{(g)} }(\XX) = \{  \omega  \wedge  \Theta(\cN_{(g)}, \Phi) | \omega  \in H^*(\XX )\}. 
\]
    Elements in  $ H_\Phi^{*+2\iota_{(g)}  }(\XX)$, called formal cohomology classes, 
can be represented by formal  products 
of closed differential forms on $\XX$ and a formal fraction of  differential forms for the fractional Thom class. These formal differential forms and their  formal cohomology classes have the 
wedge product obtained from the formal  product. We use 
the convention that integration of a  formal differential form  over $\XX$ vanishes  unless 
its formal  degree agrees with the dimension of $\XX$.

The fractional Thom class  $\Theta(\cN_e, \Phi)$ can be employed to define  a formal push-forward map 
\ba\label{Thom:Phi0}
 e_*^\Phi:      H^*(\tilde \XX)  =  \bigoplus_{(g) \in \cT_1}  H^*(\XX_{ (g)} )   \longrightarrow
 \bigoplus_{(g) \in \cT_1}   H_\Phi^{*+2\iota_{(g)} }(\XX)\na
obtained from the Thom isomorphism (\ref{Thom}) with the usual Thom class replaced by
the corresponding fractional Thom class.  It is clear from the definition that $e_*^\Phi$ is injective.  
      
  Similarly, we can define another formal   push-forward map 
\ba \label{Thom:Phi1}
 e_*^{\Phi^{-1}}:      H^*(\tilde \XX)  =  \bigoplus_{(g) \in \cT_1}  H^*(\XX_{ (g)} )   \longrightarrow
 \bigoplus_{(g) \in \cT_1}   H_{\Phi^{-1}} ^{*+ 2n_{(g)} - 2\iota_{(g)} }(\XX)
\na
with the  fractional Thom class     $\Theta(\cN_{(g)}, \Phi)$    replaced by
\[
\Theta(\cN_{(g)}, \Phi^{-1}) =  \bigwedge_{\theta_{(g)} } \big(\Theta (\cN (\theta_{(g)} )\big)^{ 1- \theta_{(g)}}.
 \]

The following lemma explains the role of these two formal push-forward maps.

 \begin{lemma}\label{lemma:identity} The Thom class of the normal bundle $\cN_{(g)}  \to \XX_{(g)} $ is given by
$
\Theta(\cN_{(g)}, \Phi )  \wedge  \Theta(\cN_{(g)}, \Phi^{-1}). 
$
  Moreover,   For $\a  \in H^*(  \XX_{(g)} )$ and  $\b \in H^*(  \XX_{(g^{-1})} )$, we have
\[
\int_{\XX  }^{orb} e_*^\Phi (\a) \wedge e_*^{\Phi^{-1}} (I^* \b) = \int_{ \XX_{(g)}}^{orb} \a\wedge I^* \b.
\]
\end{lemma}
\begin{proof}  Note that $\cN_{(g)} = \bigoplus_{\theta_{(g)}   } \cN  (\theta_{(g)}) $, 
 from which we get
  \[
\Theta(\cN_{(g)}, \Phi) = \bigwedge_{\theta_{(g)} } \big(\Theta (\cN (\theta_{(g)} )\big)^{\theta_{(g)}}, \qquad 
  \Theta(\cN_{(g)}, \Phi^{-1}) = \bigwedge_{\theta_{(g)} } \big(\Theta (\cN (\theta_{(g)} )\big)^{1- \theta_{(g)}}.
    \]
    Hence, $\Theta(\cN_{(g)}, \Phi)   \wedge   \Theta(\cN_{(g)}, \Phi^{-1})   =  \bigwedge _\theta \big(\Theta (\cN_e)\big) =  \Theta(\cN_e).$
 From the definition of  orbifold integration, we have
  \[\begin{array}{lll}
&&    \disp{ \int_{\XX }^{orb} e_*^\Phi (\a) \wedge e_*^{\Phi^{-1}} (I^* \b)}  \\[4mm]
  &=& \disp{ \int_{\cN_{(g)} }^{orb} } \pi_{(g)} ^*\a \wedge  \pi_{(g)}^* I^*  \b \wedge  \Theta(\cN_{(g)}, \Phi) \wedge  \Theta(\cN_{(g)}, \Phi^{-1}) \\[4mm]
  &=& \disp{ \int_{\cN_{(g)} }^{orb} } \pi_{(g)}^*\a\wedge  \pi_{(g)}^* I^* \b \wedge  \Theta(\cN_{(g)})\\[4mm]
  &=&\disp{  \int_{ \XX_{(g)}}^{orb} }  \a\wedge  I^*\b. 
  \end{array}
  \] 
  Here $\pi_{(g)}:  \cN_{e_{(g)} } \to \XX_{(g)}$ is the bundle projection  and $  \cN_{e_{(g)} } $ is identified with a neighbourhood of $\XX_{(g)}$ in $\XX$. 
  \end{proof}

  \begin{theorem} \label{deRham}
   Under the formal push-forward map 
\[
e_*^\Phi:  H^*(\tilde \XX)   \longrightarrow \bigoplus_{(g) \in \cT_1}   H_\Phi^{*+2\iota_{(g)} }(\XX),
\]
 the Chen-Ruan product  becomes  the wedge product.  
   That is,   for $\a, \b \in  H^{* } (\tilde\XX )$,  
  \[
  e_*^\Phi (\a *_{CR} \b) = e_*^\Phi (\a) \wedge  e_*^\Phi (\b).
  \]
\end{theorem}
\begin{proof}  From the definition of the wedge product, 
$e_*^\Phi (\a) \wedge  e_*^\Phi (\b) $  is supported near  $\tilde\XX^{[2]}$, or near the 0-section of the normal bundle $\cN$ over $\tilde\XX^{[2]}$. 

Given $(g_1, g_2) \in \cT_2$, let $e_1, e_2$ and $e_3$  be the orbifold embeddings  of
$ \XX_{(g_1, g_2)}$ in $\XX_{(g_1)}$, $\XX_{(g_2)}$ and $\XX_{(g_3)}$ respectively, where 
$g_3 = (g_1g_2)^{-1}$.    Let  $\pi_{(g)}$ be the projection $\cN_{(g)} \to \tilde \XX_{(g)}$. 
From   Theorem \ref{theorem:2}, we know that 
    the Thom  class  $\Theta(E^{[2]}_{(g_1, g_2) })$
   satisfies
  \[
  \Theta (E^{[2]}_{g_1} ) \wedge  \Theta (\cN) =  \Theta( \cN_{(g_1)}, \Phi) \wedge  \Theta(\cN_{(g_2)}, \Phi) \wedge    \Theta(\cN_{(g_3)}, \Phi),
  \]
   as a formal cohomology class  supported   near a neighbourhood of $ \XX^{[2]}_{(g_1, g_2)}$ in $\XX$.  
For $\a \in H^*(\XX_{(g_1)})$, $\b \in H^*(\XX_{(g_2)})$ 
and $ \c  \in H^*_c(\XX_{(g_3)})$,   near $\XX_{(g_1, g_2)}$ we have 
    \[
 \begin{array}{lll}
 && \disp{ \int_{\XX }^{orb} }   e_*^\Phi (\a) \wedge  e_*^\Phi (\b) \wedge  e_*^\Phi (\c)    \\[4mm]
 &=&  \disp{ \int_{\XX }^{orb} }  \pi_{(g_1)}^* (\a) \wedge  \pi_{(g_2)}^*(\b ) \wedge \pi_{(g_3)}^*(\c ) \wedge \Theta( \cN_{(g_1)}, \Phi )  \wedge  \Theta( \cN_{(g_2)}, \Phi ) \wedge  \Theta( \cN_{(g_3)}, \Phi) \\[4mm]
  &=& \disp{ \int_{\XX }^{orb} }      \pi_{(g_1)}^* (\a) \wedge  \pi_{(g_2)}^*(\b ) \wedge \pi_{(g_3)}^*(\c ) \wedge   \Theta  (E^{[2]}_{(g_1, g_2)})    \wedge\Theta(\cN )  \\[4mm]
  &=&  \disp{ \int_{\XX_{(g_1, g_2)} }^{orb} }   e_1^* (\a)  \wedge e_2^*(\b) \wedge e_3^*(\c) \wedge  e (E^{[2]}_{(g_1, g_2)}  ).
  \end{array}
  \]
 By Poinc\'are duality for orbifolds and the 3-point function  (\ref{3-point}) for the Chen-Ruan product, we get 
 \[
   e_*^\Phi (\a *_{CR} \b) = e_*^\Phi (\a) \wedge  e_*^\Phi (\b).
   \]

\end{proof}

\section{Stringy product on the orbifold K-theory}

In this section, we will define a stringy product on the orbifold K-theory of an almost compact orbifold $\XX$,  and establish a ring isomorphism between the orbifold K-theory of   $\XX$ and the Chen-Ruan cohomology ring using a modified version of the delocalized Chern character. We  first recall the definition of the delocalized Chern character and its properties. Then we give a geometric definition of  a stringy product on $K_{orb}^*(\XX)\otimes \C$ itself and show that  this product agrees with the Chen-Ruan product under a modified version of  the
delocalized Chern character. 

\subsection{Review of the Adem-Ruan-Zhang  stringy   product}\label{section:ARZ}
 
In \cite{ARZ}, Adem, Ruan and Zhang defined a stringy product on the twisted K-theory of the inertia orbifold $\tilde \XX$ of a compact, almost complex orbifold $\XX$.  This product will be called the 
Adem-Ruan-Zhang   product, denoted by  $\bullet_{ARZ}$.  To simplify  the construction, we first  discuss the untwisted case for a presentable, compact, almost complex orbifold $\XX$. 

The Adem-Ruan-Zhang   product on the orbifold K-theory of  the inertia orbifold $\tilde\XX$, under the canonical isomorphism  for  a presentable orbifold $\XX$
\[
K_{orb}^*(\tilde \XX) \cong K^*(\tilde\cG),
\]
 is  given by the formula
\ba\label{ARZ:product}
 \alpha \bullet_{ARZ} \beta =  ( e_{12})_* \bigl(e_1^*\alpha \cdot  e_2^*\beta \cdot \lambda_{-1}(E^{[2]}) \bigr), 
\na
for $\alpha, \beta \in  K^*(\tilde\cG)$, where  $ \lambda_{-1}(E^{[2]})  
= [\Lambda_{\C}^{even}E^{[2]} ]-[\Lambda_{\C}^{odd}E^{[2]} ]
 \in K^*(\tilde \cG^{[2]})$ is the K-theoretical  Euler  class of the obstruction bundle $E^{[2]}$.  
We recall  here that  the push-forward map 
\[
(e_{12})_*:  K^*(\tilde\cG^{[2]})  \longrightarrow  K^*(\tilde\cG )
\]
is given by the composition of the Thom isomorphism 
\[
  K^*(\tilde\cG^{[2]})   \cong  K_c^* (\cN_{e_{12}}) 
\] 
and the natural extension homomorphism $K_{c}^*(\cN_{e_{12}}) \to K^*(\tilde \cG)$ obtained by identifying
the normal bundle $\cN_{e_{12}}$ as a (component-wise)  tubular neighbourhood of $\tilde\cG^{[2]}$ in $\tilde \cG$. 
 Under the decomposition 
 \[
 K_{orb}^*(\tilde \XX) =\bigoplus_{(g)} K_{orb}^*( \XX_{(g)}), 
 \]
 for $\alpha_1 \in K_{orb}^*( \XX_{(g_1)})$ and $\alpha_2 \in K_{orb}^*( \XX_{(g_2)})$  we have
  \ba\label{ARZ:decomp}
 \alpha_1 \bullet_{ARZ} \alpha_2  \in K_{orb}^*( \XX_{(g_1g_2)}). 
 \na


For an  abelian  almost complex orbifold  $\XX$ obtained from a compact action of an abelian Lie group 
on a compact manifold, Becerra and Uribe in \cite{BecUri}  established  a ring homomorphism  
\[
(K^*_{orb}(\tilde \XX), \bullet_{ARZ}) 
 \longrightarrow (H^*_{dR}(\tilde \XX, \C), *_{CR}), 
 \]
 by modifying the usual Chern character as in \cite{JKK}. 
The  decomposition  theorem of Adem-Ruan (Theorem 5.1 in \cite{AR})
implies that  this  ring homomorphism is not an isomorphism in general.

 Now  we briefly recall the definition of  twisted K-theory for orbifolds. 
 A twisting on an orbifold groupoid $\cG$ is a generalized morphism
\[
\sigma:  \cG \longrightarrow PU(H)
\]
where  $PU(H)$, viewed as  a groupoid with the unit space $\{e\}$,  is the projective unitary group $PU(H)$ of an infinite dimensional separable complex Hilbert space $H$,   with the norm  topology.    A  twisting 
$\sigma$ gives rise to a principal $PU(H)$-bundle $\cP_\sigma$  over $\cG$. Two twistings are called equivalent if their associated $PU(H)$-bundles are isomorphic.  
  
  \begin{remark} It was argued in \cite{AS} that it is not desirable to use the  norm topology on the structure group $PU(H)$,  particularly in dealing with equivariant twisted K-theory. For orbifolds and orbifold groupoids, it suffices to consider   $PU(H)$ with the norm topology in the definition
  of twistings and twisted K-theory just  as in the non-equivariant case. The main reason is that the underlying groupoid for any orbifold is proper and \'etale,  so  locally it is given by the transformation groupoid for a finite group action.  As pointed out in \cite{AS},  the structure group $PU(H)$
  with the norm topology works fine for  an almost free action of a compact Lie group, where  the underlying orbifold is  presentable. 
  \end{remark}

 Let $\cK(H)$ be space of compact operators on $H$  endowed with the norm topology, and let  $Fred(H)$ be  the space of Fredholm operators endowed with the $*$-strong topology and
$Fred^1(H)$  be the space of self-adjoint elements in $Fred(H)$. Consider the associated 
bundles
\[
Fred^i (\cP_\sigma) = \cP_\sigma \times_{PU(H)} Fred^i (H)
\]
over the groupoid $\cG$. The $\sigma$-twisted $K$-theory of $\cG$, denoted $K^i(\cG, \sigma)$ 
for $i =0, 1$,  is defined to be
the set of homotopy classes of compactly supported sections of $Fred^i(\cP_\sigma)$.
For any orbifold $\XX = (X, \cU)$,  the twisted K-theory of $\XX$  is defined to be  the
 twisted K-theory
of the associated proper \'etale groupoid $\cG[\cU]$. Then the twisted orbifold $K$-theory 
is a module over the ordinary orbifold $K$-theory  and satisfies  the Mayer-Vieroris sequence   and  the Thom isomorphism for complex orbifold vector bundles.  Moreover, the  twisted orbifold $K$-theory admits the multiplicative 
operation 
\ba\label{multi}
K^*_{orb}(\XX, \sigma_1) \otimes K^*_{orb}(\XX, \sigma_2) \longrightarrow 
K^*_{orb}(\XX, \sigma_1+\sigma_2)
\na
where the addition $\sigma_1 + \sigma_2$ is  the new twisting from the 
group homomorphism $PU(H) \times PU(H) \to PU(H)$  induced by  the Hilbert
space tensor product $H \cong H\otimes H$.   See \cite{TXL} \cite{LU}
\cite{TX}  for  more detailed discussions.

Given a twisting $\tilde \sigma$ on $ \tilde \cG$,  Adem, Ruan and Zhang defined a  
product on $K_{orb}^*(\XX,  \tilde \sigma)$  by 
\ba\label{stringy2}
\alpha \bullet_{ARZ} \beta =  ( e_{12})_* \bigl(e_1^*\alpha \cdot  e_2^*\beta \cdot \lambda_{-1}(E^{[2]}) \bigr),
\na
 for $\alpha, \beta \in  K^*_{orb}(\tilde\XX, \tilde \sigma)=K^*(\tilde\cG, \tilde \sigma)$. 
Here applying the multiplicative 
operation  (\ref{multi}) and the  $K^*_{orb}(\XX)$-module structure,   
\[
e_1^*\alpha \cdot  e_2^*\beta \cdot \lambda_{-1}(E^{[2]})
\]
is an element in $K_{orb}^*(\tilde\cG^{[2]},  e_1^*\tilde\sigma + e_2^*\tilde\sigma)$, and  
\[
(e_{12})_*:  K_{orb}^*(\tilde\cG^{[2]}, e_{12}^*\tilde\sigma) \longrightarrow K_{orb}^*(\tilde\cG, \tilde \sigma) 
\]
 is the push-forward map for  (component-wise) orbifold embeddings obtained by the composition of the Thom isomorphism in twisted $K$-theory and the natural extension homomorphism for open
 embeddings. To make sense of the expression (\ref{stringy2}), we need a canonical identification 
 \[
 K_{orb}^*(\tilde\cG^{[2]},  e_1^*\tilde\sigma + e_2^*\tilde\sigma) \cong K_{orb}^*(\tilde\cG^{[2]}, e_{12}^*\tilde\sigma)
 \]
 which is guaranteed by the transgressive property of the twisting $\tilde \sigma$, see \cite{ARZ} for details.

 \subsection{Stringy product on orbifold K-theory}\label{section:stringy}

Let $\XX$ be an almost complex compact orbifold,  and $\tilde \XX =\bigsqcup_{(g)} \XX_{(g)}$  be its inertia orbifold.  Let $\cG$ and $\tilde \cG$ be their presenting proper \'etale groupoids. The evaluation map $e=\sqcup_{(g)} e_{(g)}:     \bigsqcup_{(g)} \XX_{(g)} \to \XX$ is presented by the map $e: \tilde \cG \to \cG$.

\begin{proposition} \label{ch:Phi} There exists a canonical ring homomorphism
\[
ch_\Phi: K^*_{orb} (\tilde \XX) \longrightarrow H^{*} (\tilde \XX, \C)
\]
such that 
  the following   diagram 
  \ba\label{diagram:com}
  \xymatrix{
K^*_{orb} (\tilde \XX)   \ar[rr]^{ch_\Phi}&  & H^* (\tilde \XX, \C) \\
K^*_{orb}  ( \XX)  \ar[urr]_{ch_{deloc}  }\ar[u]^{e^*} &&}
\na
commutes.
\end{proposition} 


\begin{proof} 
Any  complex  orbifold vector bundle $\cE$ over the inertia orbifold $\tilde\XX$ admits an automorphism $\Phi$.  In terms of the associated complex vector bundle $E$  over the canonical
proper \'etale groupoid $\tilde\cG$, the action of $\Phi$ on  each fiber $E_g$ over the point
$g\in \tilde \cG_0$  is given by the action of $g$.   We  have an  eigen-bundle decomposition
of $\cE$
\[
\cE = \bigoplus_{\theta \in \Q\cap [0, 1) } \cE_\theta
\]
where  $\cE_\theta$ is a complex vector bundle over $\tilde \cG$,  on which $\Phi$ acts  as multiplication by $e^{2\pi \sqrt{-1} \theta}$.   Then 
\[
ch_\Phi (\cE) = \sum_\theta e^{2\pi \sqrt{-1} \theta} ch(\cE_\theta) 
\in   H^{ev} (\tilde \cG, \C)\cong H^{ev} (\tilde \XX, \C)
 \]
defines a homomorphism
\[
ch_\Phi:  K^0_{orb} (\tilde \XX) \longrightarrow H^{ev} (\tilde \XX, \C). 
\]
 From the definition of the delocalized Chern character, we get $ch_{deloc} = ch_\Phi \circ e^*$. Hence, the diagram (\ref{diagram:com})  commutes. 
\end{proof}

Proposition \ref{ch:Phi}  implies the following commutative diagram
 \ba\label{diagram:triangle}
  \xymatrix{
K^*_{orb} (\tilde \XX) \otimes \C   \ar[rr]^{ch_\Phi}&  & H^* (\tilde \XX, \C) \\
K^*_{orb}  ( \XX) \otimes \C  \ar[urr]_{ch_{deloc}  }\ar[u]^{e^*} &&}
\na
where $ch_{deloc}$ is an isomorphism. 

For any element $\tilde\a  \in Im (e^*) \cap  Ker (ch_\Phi)$,  
there exists $\a\in K^*_{orb} (\tilde \XX) \otimes \C$ such that $\tilde \a = e^*(\a)$ and $ch_\Phi(\tilde \a) =0$.
These   imply  that $ch_{deloc}(\a) =0$. Hence $\a =0$  as  $ch_{deloc}$ is an isomorphism.
Then $\tilde \a  =  e^*(\a) = e^*(0) =0$.   Then we obtain
\ba\label{ker:Phi}
K^*_{orb} (\tilde \XX) \otimes \C   \cong Im (e^*) \oplus  Ker (ch_\Phi). 
\na
Hence, each element  $\tilde \a \in K^*_{orb} (\tilde \XX) \otimes_\Z \C$ can be uniquely written as
\ba\label{+}
\tilde \a  = e^*\a + \b 
\na
for  a unique element $\a \in K^*_{orb}  ( \XX) \otimes_\Z \C $ and $\b \in Ker (ch_\Phi)$.  Define
\ba\label{left:inverse:-1}
e_\#:   K^*_{orb} (\tilde \XX) \otimes_\Z \C \longrightarrow   K^*_{orb}  ( \XX) \otimes_\Z \C, 
\na 
sending  $\tilde \a  = e^*\a + \b$ as in (\ref{+})  to  $\a$. Then one can check that
$e_\#$ is  a left inverse of 
$e^*$ such that 
\ba\label{left:inverse}
\big( ch_{deloc} \circ e_\#  \big) (\tilde \a) = ch_\Phi (\tilde \a)
\na
for any $\tilde \a \in  K^*_{orb} (\tilde \XX)\otimes_\Z\C$.   

With these preparations, we now can define the stringy product on the orbifold K-theory
of a compact almost complex orbifold $\XX$. For simplicity, we will denote  $K_{orb}^*(\XX)\otimes_\Z \C$ and $K_{orb}^*(\tilde \XX)\otimes_\Z \C$ by  $K_{orb}^*(\XX, \C)$ and $K_{orb}^*(\tilde \XX,  \C)$ respectively.

\begin{definition} \label{stringy:orb}
Let $\XX$ be an almost complex compact orbifold,  and $\tilde \XX =\bigsqcup_{(g)} \XX_{(g)}$  be its inertia orbifold.    The stringy product  $\circ$  on  $K_{orb}^*(\XX, \C) $ is defined by
\[
\a_1 \circ \a_2 = e_\# (e^* \a_1 \bullet_{ARZ} e^* \a_2)
\]
for $\a_1, \a_2 \in K_{orb}^*(\XX, \C)$. Here $e^* \a_1 \bullet_{ARZ} e^* \a_2$ is the
Adem-Ruan-Zhang stringy product on $K_{orb}^*(\tilde \XX,  \C)$ and
$e_\#$ is the left inverse of $e^*$. 
\end{definition}
 
Next, we define a modified version of the delocalized Chern character 
\[
\widetilde{ch}_{deloc}:  K^*_{orb}(\XX) \longrightarrow H^* (\tilde\XX, \C).
\]
 For a complex vector bundle $\cE$  over  a  compact  manifold $M$, there is a well-known 
  Chern character defect for   Thom isomorphisms  in K-theory and cohomology theory, for example  see
 Chapter III.12 in \cite{LM}, such that the following diagram 
 \[
 \xymatrix{
 K^*(M, \C) \ar[rr]^{Thom}_{\cong} \ar[d]_{ch(-)\wedge \T(\cE) } && K_c^*(\cE, \C)  \ar[d]_{ch} \\
 H^*(M, \C) \ar[rr]^{Thom}_{\cong} && H_c^*(\cE, \C )  }
 \]
commutes.    Here $\T(\cE) $  is a  characteristic class of $\cE$
\[
\T(\cE) =  \dfrac{ch \big(\lambda_{-1} (\cE) \big)}{e (\cE)}  \in H^*(M, \C),  
 \] 
associated to the formal power series 
\[
\T(x) = \dfrac{1-e^x}{x}. 
\]
That is,  if we formally split the total Chern class as $c(\cE) =\prod (1+x_j)$, then
\[
\T(\cE)  =  \prod_{i} \dfrac{1-e^{x_i}}{x_i}. 
\]

 Given an orbifold  complex vector bundle over   the
 inertia  orbifold $\tilde \XX$, let 
 \[
 \cE=\bigoplus_{m_i \in \Q\cap [0, 1)} \cE_{m_i}
 \]
 be  the eigen-bundle decomposition of the canonical automorphism  $\Phi$, 
 where $\Phi$ acts on $\cE_{m_i}$ as  multiplication by  $e^{2\pi\sqrt{-1} m_i}$.  Define 
 a cohomology class
    \ba\label{defect}
 \T (\cE, \Phi) =  \prod_{m_i} \T (\cE_{m_i}) ^{m_i} \in H^{*} (\tilde\XX, \C), \na
 where $ \T (\cE_{m_i}) ^{m_i}$ is the  characteristic class
associated to the formal power series 
$ \big( \dfrac{1-e^x}{x} \big)^m. $ Then $ \T (\cE, \Phi)$ is an invertible element  in $H^{*} (\tilde\XX, \C)$ as the degree zero component is 1.
 
 Associated to the normal bundle 
\[
\cN_{e} = \bigsqcup_{(g)\in \cT_1} \cN_{(g)} 
\]
over  $\tilde \XX =\bigsqcup_{(g)} \XX_{(g)}$, we have the cohomology class
$
 \T  (\cN_e, \Phi) $ in 
 $H^{*} (\tilde \XX)$ whose component in $ H^*(\XX_{(g)}, \C)$ is given by $\T (\cN_{(g)}, \Phi)$. 
 The intrinsic description of the obstruction bundles over $\tilde\XX^{[2]}$  in Theorem \ref{theorem:2}  gives rise to the following identity
  \ba
  \begin{array}{lll} \label{key:id}
&&  \T(E^{2]}_{( g_1, g_2)} \oplus \cN_{(g_1, g_2)})  \wedge e_{12}^*  \T (\cN_{(g_1g_2)}, \Phi) \\
 & = &   e_1^*  \T (\cN_{(g_1)}, \Phi) \wedge e_2^*  \T (\cN_{(g_2)}, \Phi)
 \end{array}  \na
  in $H^*(\XX_{(g_1, g_2)}, \C)$ for any connected component $\XX_{(g_1, g_2)}$ of $\tilde\XX_{[2]}$.  Here $\cN_{(g_1, g_2)}$ is the normal bundle for  the orbifold embedding
  $e_{12}: \XX_{(g_1, g_2)} \to \XX_{(g_1g_2)}$. 
 
 \begin{definition} \label{modif:chern}   The modified delocalized Chern character on the orbifold K-theory $K^*_{orb}(\XX) \cong K^*(\cG)$ is defined to be 
 \[
\widetilde{ ch} _{deloc} =  \T (\cN_e, \Phi) \wedge ch_{deloc}:  K^*_{orb}(\XX)  \longrightarrow     H^*(\tilde \XX, \C)
\]
 \end{definition}

 Now we can prove the main result of this paper. 

\begin{theorem}  \label{ring:iso}
Let $\XX$ be a compact,  almost complex, effective orbifold. The  modified delocalized  Chern character
\[
\widetilde{ch}_{deloc}:  (  K^*_{orb}(\XX,   \C) , \circ )    \longrightarrow  (H^* (\tilde \XX, \C), *_{CR}), 
\]
is a ring isomorphism between two $\Z_2$-graded multiplicative cohomology theories. 
\end{theorem}
    \begin{proof}   By  the  definition of stringy product (Definition \ref{stringy:orb}) and  (\ref{ker:Phi})-(\ref{left:inverse}), we get 
\ba
\label{deloc:1}
\begin{array}{lll}
&&  ch_{deloc} (\a_1\circ \a_2)  \\
&=&( ch_{deloc} \circ e_\# ) \big( e^* \a_1 \bullet_{ARZ}  e^*\a_2 \big) \\
&=&  ch_\Phi \big( e^* \a_1 \bullet_{ARZ}  e^*\a_2 \big)  \\
& = &    ch_\Phi \big[ (e_{12})_*  \big( e_1^* e^* \a_1 \cdot   e_2^* e^*\a_2 \cdot    \lambda_{-1}(E^{[2]})  \big) \big]
 \end{array} 
\na
for $\a_1, \a_2 \in K^*_{orb}(\XX,   \C)$.  As $e^* =\oplus_{(g)} e_{(g)}^*$, the $\XX_{(g)}$- component of  $ch_{deloc} (\a_1\circ \a_2)$ in  (\ref{deloc:1})  is given by
\ba\label{deloc:2}
\sum_{\{ (g_1, g_2) : (g_1g_2) =  (g)\}}  ch_\Phi \big[ (e_{12})_*  \big( e_1^* e_{(g_1)}^* \a_1 \cdot   e_2^* e_{(g_2)}^*\a_2 \cdot    \lambda_{-1}(E_{(g_1, g_2)}^{[2]})  \big) \big].
\na
Here the pushforward map $(e_{12})_*: K_{orb}(\XX_{(g_1, g_2)}, \C) \to
K_{orb}(\XX_{(g_1 g_2)}, \C)$ is obtained by the composition of the Thom isomorphism for the normal bundle $\cN_{(g_1, g_2)}$ of $e_{12}: \XX_{(g_1, g_2)} \to \XX_{(g_1g_2)}$ and the natural extension for open embeddings. 
Using   the fact  that the automorphism on the Thom class of  $\cN_{(g_1, g_2)}$ is trivial, we   obtain
\[
\begin{array}{lll}
&& \widetilde{ch}_{deloc}  (\a_1\circ \a_2)   \\[2mm]
&=&  ch_{deloc} (\a_1\circ \a_2)  \wedge \T(\cN_e, \Phi)\\[2mm]
&=& \sum_{\{ (g_1, g_2)  \}}  ch_\Phi \big[ (e_{12})_*  \big( e_1^* e_{(g_1)}^* \a_1 \cdot   e_2^* e_{(g_2)}^*\a_2 \cdot    \lambda_{-1}(E_{(g_1, g_2)}^{[2]}  ) \big)  \big]   \wedge \T(\cN_{(g_1 g_2)}, \Phi) \\[2mm]
&=&  \sum_{\{ (g_1, g_2)  \}}  (e_{12})_*  \big[    e_1^* ch_\Phi  ( e_{(g_1)}^* \a_1 )\wedge   e_2^*
ch_\Phi (e_{(g_2)}^*\a_2)  \wedge  ch \big( \lambda_{-1}(E_{(g_1, g_2)}^{[2]}  \big) \\[2mm]
&&  \wedge 
\T (\cN_{(g_1, g_2)} ) \big] \wedge  \T(\cN_{(g_1 g_2)}, \Phi) \\[2mm]
&=& \sum_{\{ (g_1, g_2)  \}}  (e_{12})_*  \big[    e_1^* ch_\Phi  ( e_{(g_1)}^* \a_1 )\wedge   e_2^*  ch_\Phi (e_{(g_2)}^*\a_2)  \wedge e(E_{(g_1, g_2)}^{[2]}) \\[2mm]
&&  \wedge     \T \big(E_{(g_1, g_2)}^{[2]}\oplus\cN_{(g_1, g_2)}   \big) \wedge e_{12}^*\T(\cN_{(g_1 g_2)}, \Phi)  \big]\\[2mm]
&=&  \sum_{\{ (g_1, g_2)  \}}  (e_{12})_*  \big[    e_1^* ch_\Phi  ( e_{(g_1)}^* \a_1 )\wedge   e_2^*  ch_\Phi (e_{(g_2)}^*\a_2)  \wedge e(E_{(g_1, g_2)}^{[2]}) \\[2mm]
&& e_1^* \T (\cN_{(g_1)}, \Phi) \wedge e_2^*  \T (\cN_{(g_2)}, \Phi) \big] \\[2mm]
&=&   (e_{12})_*  \big[     e_1^* \widetilde{ch}_{deloc}  (   \a_1 )\wedge   e_2^*   \widetilde{ch}_{deloc} ( \a_2)  \wedge e(E_{(g_1, g_2)}^{[2]}) \\[2mm]
&=&  \widetilde{ch}_{deloc}  (   \a_1 ) *_{CR}  \widetilde{ch}_{deloc}  (   \a_2 ).
\end{array} 
\]
Here we apply the identity (\ref{key:id}). 
This implies that  $ \widetilde{ch}_{deloc} $ is a ring homomorphism.

From Proposition \ref{deloc:iso}, we know that $ch_{deloc}:  K^*_{orb}(\XX) \otimes \C  \to   H^*_{CR} (\XX,  \C)$ is an isomorphism of  complex vector spaces.  As the  degree zero  component of $\T(\cN_e, \Phi)$
is $1$, $\T(\cN_e, \Phi)$ is an invertible element of the ring $H^*_{CR} (\XX,  \C)$. So $ \widetilde{ch}_{deloc} $ is also  an isomorphism of  complex vector spaces.   Hence  $ \widetilde{ch}_{deloc} $ is a ring
isomorphism.

    \end{proof}
     
\begin{remark} \label{simple:def} Note that $K^*_{orb} (\tilde \XX)  =\bigoplus_{(g)}K^*_{orb}(\XX_{(g)})$, $
H^*(\tilde \XX, \C) = \bigoplus_{(g)}H^* (\XX_{(g)}, \C)$, and 
\[
ch_\Phi: \bigoplus_{(g)}K^*_{orb}(\XX_{(g)}, \C) \longrightarrow \bigoplus_{(g)}H^* (\XX_{(g)}, \C)
\]
preserves the decompositions.   This decomposition preserving homomorphism motivates the following alternative definition of stringy product which might  simplify computation when applying Theorem \ref{ring:iso}.

The isomorphism $ch_{deloc}: K^*_{orb}(\XX,  \C)   \longrightarrow  \bigoplus_{(g)}  H^{*}_{orb}(  \XX_{(g)} , \C) $ induces the decomposition 
  \ba\label{decom:0}
  K^*_{orb}(\XX,   \C) =  \bigoplus_{(g)}  K^{*}_{orb}\big(  \XX, (g) \big)
  \na
  such that $ch_{deloc}$ preserves the decomposition. 
  Then the commutative traingle (\ref{diagram:triangle}) immediately implies the following identities
  \ba\label{decom}
  ch_{\Phi}  (e_{(h)}^* \omega_{(g)}  )  = \left\{\begin{array}{lll}
   ch_{deloc}  (  \omega_{(g)}  ) &\qquad & \text{if $(g) =(h)$} \\
   0   &\qquad & \text{otherwise } \end{array}\right.
\na
for any $\omega_{(g)} \in  K^{*}_{orb}\big(  \XX, (g) \big)$,  and  where $e_{(h)}: \XX_{(h)} \to \XX$ denotes  the orbifold embedding. 
Define 
\[
\bar{e}^* =\bigoplus_{(g)} e^*_{(g)}:  K^*_{orb}  ( \XX,  \C) = \bigoplus_{(g)}  K^{*}_{orb}\big(  \XX, (g) \big) \longrightarrow
K^*_{orb} (\tilde \XX,  \C)  =\bigoplus_{(g)}K^*_{orb}(\XX_{(g)},  \C)
\]
by sending $\a_{(g)} \in K^{*}_{orb}\big(  \XX, (g) \big)$ to $e_{(g)}^*\a_{(g)} \in K^*_{orb}(\XX_{(g)},  \C)$.   
It is straightforward to check  that $\bar{e}^*$ is a  ring homomorphism preserving the decomposition. Moreover,
  the homomorphism $\bar{e}^*$ satisfies 
\[
ch_{deloc} = ch_\Phi \circ \bar{e}^*
\]
and  has a canonical   left inverse $\bar{e}_\#$ defined in a  similar way as     in (\ref{+})  - (\ref{left:inverse}).  We can check that the stringy product in Definition \ref{stringy:orb} can be written as 
\ba\label{new:def}
\a_1 \circ \a_2 = \bar{e}_\# (\bar{e}^* \a_1 \bullet_{ARZ} \bar{e}^* \a_2)
\na
for $\a_1, \a_2 \in K_{orb}^*(\XX, \C)$.   In particular, if $\a_1 \in K_{orb}^*\big(\XX, (g_1) \big)$ and
$\a_2 \in K_{orb}^*\big(\XX, (g_2) \big)$, then
\[
\a_1 \circ \a_2 \in K_{orb}^*\big(\XX, (g_1g_2) \big).
\]
  \end{remark}

  \begin{example}  \label{1} 
The first  example   is the  orbifold  $[\{pt\}/G]$ given by a point $\{pt\}$
  with a trivial action of a finite group $G$. 
   The   groupoid $\cG$ associated to   $[\{pt\}/G]$ is the action groupoid $G\rightrightarrows \{pt\}$ with   the conjugation action groupoid $G \times G \rightrightarrows   G$ as  its  inertia groupoid $\tilde \cG$.

  The Chen-Ruan cohomology ring of  $[\{pt\}/G]$  is isomorphic to the center  $Z(\C[G]) $ of the group algebra $\C[G]$, see \cite{CR} or  \cite{ALR}.  Alternatively,   $H^*(\tilde \cG, \C)$  
 can be identified with the  vector space of $\C$-valued $G$-conjugation invariant functions (also called class functions) on $G$.   Let $\cC(G)$ be the vector space of $\C$-valued class functions on $G$.  The ordinary cup product on  $H^*(\tilde \cG, \C)$ is the point-wise multiplication on $\cC(G)$.  As complex vector spaces,   $\cC(G)$     is  isomorphic to the center  $Z(\C[G]) $ of the group algebra $\C[G]$.  The Chen-Ruan product  on   $Z(\C[G]) $  induces the  convolution product on $\cC(G)$
  \ba\label{convol}
( \chi_1 * \chi_2 )(g) =\sum_{\{g_1, g_2 \in G: g_1g_2 = g\}}  \chi_1(g_1) \chi(g_2)
\na
for $\chi_1, \chi_2 \in \cC(G)$, and $g\in G$. This means
\[
\big(H^*_{CR}( [\{pt\}/G], \C), *_{CR}\big) \cong  (\cC(G), *).
\]
  
The stringy product  on   the  orbifold K-theory of the inertia orbifold  was described in \cite{ARZ}.  Here we consider the stringy  product on $K_G(\{pt\})$, the  orbifold K-theory of the
orbifold  $[\{pt\}/G]$ itself. 
Note that 
\[
K_G(\{pt\}) \cong R(G), 
\]
 the representation ring of $G$.  The odd equivariant K-theory of a point vanishes.
Let   $  \rho_V: G\to GL(V)$  be a representation  of $G$. Then from  the definition,  
 $ch_{deloc} ([V, \rho_V])$ is a   class function 
  on $G$ given by the character of the representation $\rho_V$
\[
ch_{deloc} ([V, \rho])  (g) = Tr \big(\rho_V(g)\big), \qquad \text{for} \qquad g\in G.
\]
that is, $ch_{deloc}: R(G) \to \C$ is a $\C$-valued class function on $G$.   Then 
\[
ch_{deloc}:  R(G)\otimes \C  \longrightarrow \cC(G)
\]
a ring isomorphism, where the ring structure on $R(G)\otimes \C $ is given by the  tensor product   of representations, and the ring  structure on   $\cC(G)$ is the standard point-wise multiplication.

Let $  \rho_V: G\to GL(V)$ and
$  \rho_W: G\to GL(W)$ be two representations of $G$.  We compute the Adem-Ruan-Zhang
product  
\[
e^* [V, \rho_V] \bullet_{ARZ}  e^*[W, \rho_W] = (e_{12})_* \big(
e_1^*e^* [V, \rho_V] \otimes  e_2^*e^*[W, \rho_W]  \big)
\]
represented by a $G$-equivariant vector bundle over $G$, since  there are no normal bundles  involved in this case.  As a vector bundle over $\tilde \cG$, its  fiber over $g\in G$ is 
\[
  \bigoplus_{\{g_1, g_2 \in G: g_1g_2 = g\}} V_{g_1} \otimes W_{g_2}
\]
where $V_{g_1} = V$ and $W_{g_2} = W$. The  automorphism $\Phi$ on the fiber over $g$   is given by 
\[
 \bigoplus_{\{g_1, g_2 \in G: g_1g_2 = g\}} \rho_V(g_1) \otimes 
\rho_W(g_2).
\] 
By a direct calculation, we get
\[
ch_\Phi (e^* [V, \rho_V] \bullet_{ARZ}  e^*[W, \rho_W] ) = ch_{deloc} ([V, \rho_V]) * ch_{deloc} ([W, \rho_W])
\]
where $*$ is the convolution product on $\cC(G)$. 
Therefore, the stringy product on $K_{orb}([\{pt\}/G], \C)$ agrees with  the convolution product on $\cC (G)$ under the delocalized Chern character.  Therefore,  we get
the ring  isomorphism
\[
 \big(K_{orb}^*([\{pt\}/G]) \otimes \C, \circ \big) \cong   \big(H^*_{CR}([\{pt\}/G], \C), *_{CR}\big).
\]

In summary, for the orbifold $[\{pt\}/G]$, the orbifold K-theory $K_{orb}^*( [\{pt\}/G]) \cong R(G)$ is isomorphic to the singular  cohomology of its inertia orbifold  over the complex coefficients under
the delocalized Chern character. There are two products
on $K_{orb}^*( [\{pt\}/G], \C)$. 
\begin{enumerate}
\item One is given by the tensor product, under the delocalized Chern character, corresponding to the point-wise product on $\cC(G)$. 
\item The other is the stringy product,  under the delocalized Chern character, corresponding to the  convolution product on $\cC(G)$. 
\end{enumerate}

  \end{example}

\begin{example} The orbifold is $[G/G]$,  obtained by the conjugation action of a finite group $G$ on itself.
The groupoid  $\cG$ for $[G/G]$ is the   action groupoid of the $G$-conjugation action on $G$. Its inertia
groupoid $\tilde \cG$  is the   action groupoid of the  $G$-conjugation action on
\[
\tilde G= \{(g, h)\in G\times G|  gh = hg\}, 
\]
  the set of pairs of commuting elements in $G$. 
  
   Additively,   $K_{orb}^*([G/G])$,  the orbifold K-theory of $[G/G]$,   is isomorphic to 
   \[
   K_G(G)\cong 
   \bigoplus_{\{(g): \text{conjugacy classes of $G$}\}} R(Z_G(g)),
   \]
   where $Z_G(g)$ is the centralizer of $g$ in $G$, see \cite{ARZ}.  An element in $K_G(G)$ is
   represented by a collection of finite dimensional complex vector spaces 
   $ \{(V_g, \rho_g)\}_ {g\in G}$,  where  $\rho_g:  Z_G(g) \to GL(V_g)$ is a group homomorphism, such that 
   there is a linear $G$-action on $\oplus_g V_g$ intertwining with the $\oplus \rho_g$-representation in the following sense: 
   \begin{enumerate}
\item  for any $k\in G$, there is a linear isomorphism $\phi_V(k):  V_g \to V_{kgk^{-1}}$ satisfying  $\phi_V(k_1k_2)= \phi_V(k_1)\circ \phi_V(k_2)$ for $k_1, k_2 \in G$, 
 \item $\phi_V(h) = \rho_g(h)$ for  $h\in Z_G(g)$, 
 \item   the following diagram commutes  
   \ba\label{intertwining}  \xymatrix{
   V_g \ar[rr]^-{\rho_g (h)} \ar[d]_{\phi_V(k)}^\cong&& V_g\ar[d]^{\phi_V(k)}_\cong\\
   V_{kgk^{-1} } \ar[rr]_-{\rho_{kgk^{-1}}(khk^{-1})} && V_{kgk^{-1}}
 }\na
for any $h\in Z_G(g)$. 
\end{enumerate}
The diagram (\ref{intertwining})    implies that for $(g, h) \in \tilde G$
 \[
 Tr\big(\rho_g(h)\big) = Tr \big(\rho_{kgk^{-1}}(khk^{-1}\big).
 \]
 Hence  the function $\chi_{\{(V_g, \rho_g)\}_ {g\in G}}: (g, h) \mapsto  Tr\big(\rho_g(h)\big) $ is a $G$-invariant function
 on $\tilde G$.      Let $\cC_G(\tilde G)$ be the  
   space of $\C$-valued $G$-invariant functions on $\tilde G$. 
    Then
   the delocalized Chern character  
   \[
   ch_{deloc}: K_G(G) \longrightarrow \cC_G(\tilde G)
   \]
   is given by  sending $\{(V_g, \rho_g)\}_ {g\in G}$ to  $\chi_{\{(V_g, \rho_g)\}_ {g\in G}}$.

   There are three different products on $K_G(G) \otimes_\Z \C$ described as follows. 
  
  \begin{enumerate}
\item 
The first one is  given by the usual
 tensor product $\otimes_G$ of $G$-equivariant vector bundles over $G$.   The delocalized Chern character
  \[
  ch_{deloc}:  (K_G(G)   \otimes \C, \otimes_G)  \longrightarrow  (\cC_G(\tilde G), \cdot)
  \]
  is a ring isomorphism for the point-wise product $\cdot$  on $\cC_G(\tilde G)$.
  
 \item The second one is the Pontryagin product  $\bullet_G$ given by 
 \[\xymatrix{
\bullet_G:   K_G(G) \times K_G(G)  \ar[rr]^-{\pi_1^*\times \pi_2^*} &&
 K_G(G\times G) \ar[r]^-{m_*} & K_G(G),}
 \]
 where $\pi_1, \pi_2: G\times G \to G$ are the obvious projections and  $m: G\times G \to G$ is the group
 multiplication.     This product $\bullet_G$ agrees with the Adem-Ruan-Zhang
product $\bullet_{ARZ}$ on the orbifold K-theory of the inertia orbifold of $[\{pt\}/G]$. 
 The ring $(K_G(G), \bullet_G)$, or its complexification, is   the 
fusion  ring for the three-dimensional topological quantum field theory associated to the finite gauge group $G$, see \cite{Freed1} and \cite{Freed2}.

Explicitly,  given two elements 
 \[
 V = \oplus_{(g)} [V_{(g)}],  W = \oplus_{(g)} [W_{(g)}]  \]
 in $K_G(G) =  \oplus_{(g)} R(Z_G(g))$,  the Pontryagin product  is given by
  \[
 \big\{(V\bullet_G W)_g =  \bigoplus_{g_1, g_2\in G: g_1g_2 =g} V_{g_1}\otimes W_{g_2}\big\}_{g\in G}
 \]
 with the diagonal  $G$-action 
 \[
  \bigoplus \phi_V (k) \otimes \phi_W (k):
 (V\bullet_G W)_g \longrightarrow 
 (V\bullet_G W)_{kgk^{-1}}
  \]
 for any $k \in G$.   The delocalized Chern character
  \[
  ch_{deloc}:  K_G(G)   \otimes \C  \longrightarrow  \cC_G(\tilde G) 
  \]
  is a linear isomorphism of vector spaces.   For an abelian group $G$, 
  $\tilde G = G\times G$. Then the  Pontryagin product  $\bullet_G$  induces   the convolution product  on $\cC_G(\tilde G) $  in the first variable
  \[
  (f_1\ *_1  f_2 )(g, h) = \sum_{\{ g_1, g_2 \in  G:  g_1g_2 = g \}} f_1(g_1, h) f_2(g_2, h).
\]

 \item The third one is the stringy product  $\circ$ on the orbifold K-theory of the orbifold $[G/G]$ itself defined in Definition  \ref{stringy:orb}.  Given two elements 
 \[
 V = \oplus_{(g)} [V_{(g)}],  W = \oplus_{(g)} [W_{(g)}]  \]
 in $K_{orb}^*([G/G], \C)  \cong   \oplus_{(g)} R(Z_G(g)) \otimes \C$, 
 the stringy  product  on $[G/G]$ is given by
 \[
 V\circ W = \oplus_{(g)} [V_{(g)} \circ_{Z_G(g)} W_{(g)}],
 \]
 where $V_{(g)} \circ_{Z_G(g)} W_{(g)}$ is the stringy product on $K^*_{orb}([\{pt\}/Z_G(g)], \C)$ discussed in Example \ref{1}.
 
  The delocalized Chern character
  \[
  ch_{deloc}:  (K_G(G)   \otimes \C, \circ )  \longrightarrow  (\cC_G(\tilde G), *_2)
  \]
  is a ring isomorphism, where    the  product $*_2$  on $\cC_G(\tilde G)$ is given by
  \ba\label{hidden:product}
  (f_1*_2 f_2 )(g, h) = \sum_{\{ h_1, h_2 \in   Z_G(g):  h_1h_2 = h \}} f_1(g, h_1) f_2(g, h_2).
  \na
 For  an abelian group $G$,  $*_2$  is the convolution product  on $\cC_G(\tilde G) $ 
  in the  second  variable
  \[
 (f_1*_2  f_2 )(g, h) = \sum_{\{ h_1, h_2 \in   G:   h_1h_2 = h \}} f_1(g, h_1) f_2(g, h_2).
\]
  \end{enumerate}
A  simple example like $G= \Z_2$ shows that these three products are indeed different. 
  
 \end{example}

\begin{corollary}   For any finite group $G$, there is a ring isomorphism  
\[
(H_{CR}([G/G], \C), *_{CR}) \cong  (\cC_G(\tilde G), *_2).
\]
\end{corollary}
 \begin{example}  (Weighted projective spaces)
Consider the weighted projective space $W\PP (p,q)$, where $p$ and $q$ are coprime integers,
which can be presented   as the quotient of the unit sphere $S^3 \subset \C^2$ by  the $S^1$-action
\[
e^{i\theta} (z_1, z_2) = (e^{ip\theta}  z_1, e^{iq\theta} z_2).
\]
As an orbisphere, $W\PP (p,q)$ can be covered by two orbifold charts at singular
points $x=[1,0]$ and $y=[0,1]$ with   isotropy groups $\Z_{p}$ and $\Z_{q}$ respectively.
Here $\Z_{w_i}$ is the cyclic subgroup of $S^1$ generated   by the primitive $w_i$-th root of unity.

The orbifold K-theory of $W\PP (p,q)$ over  $\C$ was given in \cite{AR} as 
\[
K_{orb}^* \big( W\PP (p,q) , \C\big)  \cong \C [u]/\<(1-u)^2\> \oplus  \C (\zeta_p) \oplus  \C (\zeta_q)
\]
where $\C (\zeta_p)$ and $\C (\zeta_q)$ are the $p$-th and $q$-th cyclotomic fields over $\C$.


We can apply Theorem \ref{ring:iso},  Remark  \ref{simple:def}  and Example 4.28 in \cite{ALR}  to
 find the stringy product on $K_{orb}^* \big( W\PP (p,q), \C\big)$. Let 
 \[
 \a_p \in \C \oplus \C(\zeta_p) \qquad \b_q \in 
 \C \oplus \C(\zeta_q) 
 \]
  be elements such that  the  delocalized Chern characters correspond to the constant function 1 on the twisted sectors corresponding to generators of  $\Z_p$ and $\Z_q$ respectively. Then 
  we have
  \[
  \a_p^p = \b_q^q = 1-u,   \a_p^{p+1} = \b_q^{q+1} =0.
  \]

 \end{example}


 \section{Twisted cases}

As explained in Section \ref{section:ARZ}, the Adem-Ruan-Zhang stringy product  is defined for twisted K-theory of the inertia orbifold of a compact, almost complex orbifold.  In this section, we explain how the constructions in Section 4 can be carried over to torsion twisted cases. 

A twisting  $\sigma$ over an  orbifold $\XX$ is a principal $PU(H)$-bundle over $\XX$. Then there exists an orbifold atlas such that $\sigma$ is represented by an $U(1)$-central 
extension of the canonically associated groupoid $\cG[\cU] = (\cG_1
\rightrightarrows \cG_0).$   This 
central  extension 
\ba\label{gerbe}
\xymatrix{
U(1) \ar[r] &  \ccR \ar[r] & \cG_1
\rightrightarrows \cG_0,   }
\na
called a $U(1)$-gerbe over $(\cG_1 \rightrightarrows \cG_0)$, 
is obtained from the $PU(H)$-valued cocycle on $\cG[\cU] $ and the 
$U(1)$-central 
extension of  $PU(H)$. 
Note that a gerbe connection is a connection $\theta$ on the $U(1)$-bundle $\ccR$
which is compatible with the groupoid multiplication on $\ccR$.  A curving for the connection
$\theta$ is a 2-form $B$ on $\cG_0$ such that the curvature of $\theta$
\ba\label{gerbe:conn}
F_\theta = s^* B - t^*B \in \Omega^2(\cG_1).
\na
Then $dB \in \Omega^3(\cG_0)$ satisfies the property
\[
 s^* (d B)  - t^* (dB ) =0.
 \]
 Hence, the 3-form  $dB$ defines a closed 3-form $\Omega$ on the orbifold $\XX$, called the
 gerbe curvature of the gerbe $\ccR$ with a gerbe connection $\theta$ and a curving $B$. 
 A twisting $\sigma$  is called torsion if its associated gerbe is flat, that is,  it has a gerbe connection and a curving whose   gerbe curvature vanishes.
By a  standard construction (Proposition 3.6 in   \cite{TX}), taking a refined orbifold atlas
 if necessary,   the  gerbe  (\ref{gerbe})  always admits a gerbe connection and a curving.

The central extension (\ref{gerbe}) defines a complex line bundle  over $\cG_1$.  Restricted to
the inertia groupoid $\tilde \cG$,  we get   a complex line bundle over the inertia groupoid, denoted by
\ba\label{line:bundle}
\cL_\sigma = \sqcup_{(g)\in \cT_1}  \cL_{(g)} \longrightarrow  \sqcup_{(g)\in \cT_1}  \XX_{(g)}.
\na
When the gerbe is equipped with a gerbe connection and curving, the property (\ref{gerbe:conn}) implies that
the induced connection on the complex line bundle $\cL_\sigma$ is flat.  Hence, $\cL_\sigma =\{ \cL_{(g)}\}_{(g)\in \cT_1}$ is an inner local system defined in \cite{PRY} and \cite{Ruan4}.  Recall that an inner local system over an orbifold is a flat 
complex orbifold line bundle 
\[
\cL_\sigma= \sqcup_{(g)\in \cT_1}  \cL_{(g)} \longrightarrow  \sqcup_{(g)\in \cT_1}  \XX_{(g)}\]
satisfying the following compatibility  conditions
\begin{enumerate}
\item $\cL_{(1)}$  is a trivial orbifold line bundle with a fixed trivialization.
\item There is a non-degenerate pairing $\cL_{(g)} \otimes I^* \cL_{(g^{-1})} \to    \cL_{(1)}$.
\item There is an associative multiplication 
\[
e_{(g_1)}^* \cL_{(g_1)} \otimes   e_{(g_2)}^* \cL_{(g_2)}  \longrightarrow e_{(g_1g_2)}^* \cL_{(g_1g_2)}
\]
over $\XX_{(g_1, g_2)}$ for $(g_1, g_2)\in \cT_2$. Here $e_3 = I\circ e_{12}$. 
\end{enumerate}

Given an inner local system $\cL_\sigma$ on  an almost complex, compact orbifold $\XX$,  the $\cL_\sigma$-twisted Chen-Ruan cohomology  is defined
to be 
 \ba\label{CR:twisted}
 H_{CR}^*(\XX, \cL_\sigma) =   \bigoplus_{(g)\in \cT_1} H^{d-2\iota_{(g)}} (\XX_{(g)}, \cL_{(g)})
\na
with the twisted Chen-Ruan product given by the same expression as in Definition \ref{CR}.  The 
properties of the inner local system ensures that the Poincar\'e pairing 
\[
\< \omega_1, \omega_2 \>  =\int_{\XX_{(g)}}^{orb} \omega_1 \wedge I^*\omega 
\]
for $\omega_1 \in H^*(\XX_{(g)}, \cL_{(g)}) $ and $\omega_2 \in H^*(\XX_{(g^{-1})}, \cL_{(g^{-1})}) $ is non-degenerate, and the twisted Chen-Ruan product is determined by the 
 the 3-point function 
 \[
\< \omega_1 *_{CR} \omega_2 , \omega_3 \> = \int_{\XX_{(g_1, g_2)}} ^{orb} e_1^* \omega_1
\wedge e_2^* \omega_2 \wedge e_3^* \omega_3 \wedge e(E^{[2]}_{(g_1, g_1)} ) ,
\]
  for $\omega_1 \in H^*(\XX_{(g_1) }, \cL_{(g_1)})$, 
$\omega_2 \in H^*(\XX_{(g_2) }, \cL_{(g_2)})$ and $\omega_3 \in H^*(\XX_{(g_3) }, \cL_{(g_3)})$
with $g_3= (g_1g_2)^{-1}$.  

Let $\sigma$ be  a torsion twisting  on an almost complex compact orbifold $\XX$.   We assume that the twisted orbifold K-theory defined by the Grothendieck group of $\ccR$-twisted vector bundles over $\XX$ agrees with  the twisted K-theory defined in Section \ref{section:ARZ}.   Let $\check{\sigma}$ be the gerbe
associated to $\sigma$    with a gerbe connection and curving.  Note that if  the gerbe curvature of  $\check{\sigma}$ is zero, then the twisted Chern character on the twisted orbifold
K-theory $K_{orb}^*(\XX, \sigma)$ constructed in \cite{TX}   
\ba\label{twisted:chern}
ch_{\check{\sigma}}:    K_{orb}^*(\XX, \sigma) \otimes\C \longrightarrow H^*(\tilde \XX, \cL_\sigma) =   \bigoplus_{(g)\in \cT_1} H^{*} (\XX_{(g)}, \cL_{(g)})
 \na
 is an isomorphism of  vector spaces over the complex coefficients.  Geometrically, this  twisted Chern character is  constructed  as follows. Let $E$ be an $\ccR$-twisted vector bundle over $\cG_0$. Then the pull-back vector bundle  $e^* E $ over $\tilde \cG_0$ admits a bundle isomorphism 
 $\Phi$ defined by an element in $End(e^*E)\otimes \cL_\sigma$. With respect to a locally constant trivialization,
 the Chern character $ch_\Phi$ as in Section \ref{section:stringy}   can be extended to the twisted case to define a 
 homomorphism
 \[
 ch_{\Phi, \sigma}:   K_{orb}^*(\tilde\XX,  e^*\sigma) \longrightarrow  H^*(\tilde \XX, \cL_\sigma)
 \]
 such that the following diagram commutes
 \ba
  \xymatrix{
K^*_{orb} (\tilde \XX, e^*\sigma) \otimes \C   \ar[rr]^{ch_{\Phi, \sigma}}& 
 & H^* (\tilde \XX, \cL_\sigma). \\
K^*_{orb}  ( \XX, \sigma) \otimes \C  \ar[urr]_{ch_{ \check{\sigma}}  }\ar[u]^{e^*} &&}
 \na
 
 If  $e^*\sigma$ is transgressive,   then we can apply the Adem-Ruan-Zhang  product 
  on  $K_{orb}^*(\tilde \XX, e^*\sigma)$ to define a stringy product 
  on $K_{orb}^*(\XX, \sigma)\otimes\C$. Modifying the twisted Chern character in (\ref{twisted:chern})
  as in Definition \ref{modif:chern} and  Theorem \ref{ring:iso}, we expect that 
   $K^*_{orb}  ( \XX, \sigma) \otimes \C$  with  this  string product   is isomorphic to  the twisted Chen-Ruan cohomogy  $H^*(\tilde \XX, \cL_\sigma)$ using the Mayer-Vietoris exact sequence.   
   We leave the details of   the proof of this  isomorphism  to interested readers.

\vskip .2in
\noindent
{\bf Acknowledgments} The authors would like to thank Yongbin Ruan for many invaluable discussions during the course of this work. This work is partly supported the NSFC Grant  10825105 (Hu)  and  the ARC Discovery Grant   (Wang).   Hu thanks MSI of Australian National University and Wang thanks Sun Yat-Sen University for their hospitality during part of the writing of this paper.  The authors like to thank the referee for valuable suggestions to improve  the expositions and Alex Amenta for his comments.

  \end{document}